\documentclass[12pt]{amsart}
\usepackage{amsmath, amssymb}
\newfont {\cyr} {wncyr10}
\renewcommand{\labelenumi}{{(\roman{enumi})}}

\usepackage{amsfonts}
\usepackage{amsmath}
\usepackage{amssymb}
\usepackage{setspace}
\usepackage{multicol}

\newtheorem{theorem}{Theorem}[section]
\newtheorem{lemma}[theorem]{Lemma}\newtheorem{notation}[theorem]{Notation}

\newtheorem{corollary}[theorem]{Corollary}
\newtheorem{definition}[theorem]{Definition}

\newcounter{claim}[theorem]

\newcounter{cclaim}[theorem]

\usepackage{color}

\def \udot {{}^{\textstyle .}}

\newcommand{\E}{\mathrm{E}}\newcommand{\SU}{\mathrm{SU}}
\newcommand{\F}{\mathrm{F}}\newcommand{\B}{\mathrm{B}}\newcommand{\M}{\mathcal{M}}
\newcommand{\G}{\mathrm{G}}

\newcommand{\Q}{\mathrm{Q}}

\newcommand{\Aut}{\mathrm{Aut}}

\newcommand{\Out}{\mathrm{Out}}

\newcommand{\Syl}{\mathrm{Syl}}\newcommand{\syl}{\mathrm{Syl}}

\newcommand{\GF}{\mathrm{GF}}
\newcommand{\GL}{\mathrm{GL}}
\newcommand{\Sp}{\mathrm{Sp}}
\newcommand{\SL}{\mathrm{SL}}

\newcommand{\PSL}{\mathrm{PSL}}\newcommand{\PSp}{\mathrm{PSp}}
\newcommand{\Sym}{\mathrm{Sym}}
\newcommand{\Alt}{\mathrm{Alt}}
\newcommand{\Dih}{\mathrm{Dih}}

\newcommand{\U}{\mathrm{U}}

\def \OO {\hbox {\rm O}}

\def \syl {\hbox {\rm Syl}}\def \Syl {\hbox {\rm Syl}}

\def \wt {\widetilde}

\def \Aut{ \mathrm {Aut}}

\def \Out{\mbox {\rm Out}}

\def \J{\mbox {\rm J}}

\def \B{\mbox {\rm B}}

\def \M{\mbox {\rm M}}

\def \Co {\mbox {\rm Co}}

\def \PSU {\mbox {\rm PSU}}
\def \GSp {\mbox {\rm GSp}}\def \GO {\mbox {\rm GO}}

\begin{document}
\renewcommand{\labelenumi}{(\roman{enumi})}

\title  {${\mathbf \F_4(2)}$ and its automorphism group}
 \author{Chris Parker}
  \author{Gernot Stroth}

\address{Chris Parker\\
School of Mathematics\\
University of Birmingham\\
Edgbaston\\
Birmingham B15 2TT\\
United Kingdom} \email{c.w.parker@bham.ac.uk}

\address{Gernot Stroth\\
Institut f\"ur Mathematik\\ Universit\"at Halle - Wittenberg\\
Theordor Lieser Str. 5\\ 06099 Halle\\ Germany}
\email{gernot.stroth@mathematik.uni-halle.de}

\email {}

\date{\today}

\begin{abstract} We present an identification theorem for  the groups $\F_4(2)$ and  $\Aut(\F_4(2))$ based on the   structure of the centralizer of an element of order $3$.
\end{abstract}
\maketitle \pagestyle{myheadings}

\markright{{\sc }} \markleft{{\sc Chris Parker and Gernot Stroth}}

\section{Introduction}

In the classification of the finite simple groups a fundamental role was played by Timmesfeld's work on groups which contain a large extraspecial 2-subgroup \cite{Tim}. Timmesfeld determined the structure of the normalizer of such a subgroup and following this achievement several authors contributed to the classification of  all the  simple groups which contain a large extraspecial $2$-subgroup.

The notion of a large extraspecial $2$-subgroup of a group is generalized  in the work of Meierfrankenfeld,  Stellmacher and  the second author \cite{MSS} to the concept of a large $p$-subgroup where $p$ is an arbitrary prime. The definition of a large $p$- subgroup is as follows: given a finite group   $G$,  a $p$-subgroup $Q$ of $G$  is  \emph{large} if and only if
\medskip
\begin{enumerate}
\item[(L1)]\label{1} $Q = F^*(N_G(Q))$; and
\item[(L2)]\label{2} for all  non-trivial subgroups $U$  of $ Z(Q)$,  $N_G(U)\le N_G(Q)$.
\end{enumerate}
\medskip
Recall that  condition (L1) is equivalent to $Q=O_p(N_G(Q))$ and $C_G(Q)\le Q$. If $Q$ is extraspecial and $p = 2$ this definition coincides with  Timmesfeld's definition of a large extraspecial 2-group. The classification of groups with a large $p$-subgroup is sometimes called the MSS-project. The first step of this project is \cite{MSS}, where in contrast to the work of Timmesfeld,  it is not the normalizer of $Q$ which is determined but rather structural   information about the maximal $p$-local subgroups of $G$ which are not  contained in $N_G(Q)$ is provided.

Suppose now that $Q$ is a large subgroup of a group $G$ and let $S$ be a Sylow $p$-subgroup of $G$ containing $Q$. It is an elementary exercise to show that $F^\ast(N_G(U)) = O_p(N_G(U))$ for all non-trivial normal subgroups $ U$ of $ S$ (\cite[Lemma  2.1]{PS2}).  Groups which satisfy this property are said to be of \emph{parabolic characteristic $p$}.
If  $F^\ast(N_G(U)) = O_p(N_G(U))$ for all $1 \not= U \le S$, then $G$ is of \emph{local characteristic $p$} (also called characteristic $p$-type).
In \cite{MSS} it is assumed that $G$ has local characteristic $p$. However, there is work in progress which aims to remove  this assumption, and so all the successor articles to \cite{MSS} will be produced under the weaker hypothesis that the group under investigation has  a large $p$-subgroup.  One reason for this is that, as mentioned above, a group with a large $p$-subgroup is of  parabolic characteristic $p$, while demonstrating that a group has local characteristic $p$ may well  be hard to verify in applications.

 Nevertheless  \cite{MSS} provides us with some $p$-local structure of the group $G$ and this is all what we require for the next step of the programme in which  we aim to  recognize  $G$ up to isomorphism. For this recognition we typically build a geometry upon which a subgroup of $G$  acts. This means that we take some of the $p$-local subgroups of $G$ which contain $S$ and  consider the subgroup $H$ of $G$ generated by them. The $p$-local subgroups are  selected so that  $O_p(H) = 1$. As the generic simple groups with a large $p$-subgroup are Lie type groups in characteristic $p$, in many cases we will be able to show that the coset geometry determined by the $p$-local subgroups in $H$ is a building. The recognition of $H$ is then achieved with  help of the classification of buildings of spherical type \cite{Tits,  local}. At this stage, as a third step of the programme, we would like  to show that $G=H$. There is a general approach to achieve this goal. Since $H$  contains $S$ it also contains $Q$ and so we are able to identify $Q$ as a subgroup of $H$. Typically $Q=F^\ast(N_H(R))$ for some root group $R$ in $H$. We can then  determine the structure of $N_G(Q)$. The aim is to show that $N_G(Q) = N_H(Q)$ and from this further show that $N_G(U) = N_H(U)$ for all $1 \not= U \unlhd S$. The final step  is to show that, if $H$ is a proper subgroup of $G$, then $H$ is strongly $p$-embedded in $G$ and this contradicts the main results in   \cite{Be} and \cite{PSStrong}.

However there are situations  where it cannot be shown that $N_G(Q) = N_H(Q)$. This happens most frequently  when $p = 2$ or $3$ and $N_H(Q)$ is soluble. For the final stage of the MSS-project  one has to  analyze    exactly these more troublesome configurations; that is  determine all the groups $G$ where $F^*(H)$ is a group of Lie type in characteristic $p$ containing a Sylow $p$-subgroup $S$ of $G$, $N_H(Q)$ is soluble and $N_H(Q) \not= N_G(Q)$. There are several  configurations where this phenomenon arises. For
example  when $p=3$  we  have  $H \cong \Omega^-_6(3)$  contained in  $G \cong \U_6(2)$. Similarly, there are
containments $\Omega^+_6(3)$ in $\F_4(2)$, $\Omega_7(3)$ in  ${}^2\E_6(2)$ and $\M(22)$, and $\Omega^+_8(3)$ in
$\M(23)$ and $\F_2$. In all these cases $Q$ is an extraspecial 3-group and $N_H(Q)$ is soluble. In a series of papers \cite{PS1, PS3, PSS},  the larger groups in this list are determined from the
approximate structure of the centralizer of an element of order 3, or equivalently from the structure of $N_G(Q)$. In this paper we  identify $\F_4(2)$ from the approximate structure of the centralizer of a $3$-element. We are motivated by the embedding of $\Omega^+_6(3)$ in $\F_4(2)$, but we do not assume that $G$ contains this group. We just assume certain important structural information about the normalizer of $Q$ and, as a consequence, this present article is independent of the results in \cite{MSS}.
This article should also be viewed as a companion to the authors' earlier work \cite{PS1} in which the groups $G$ with $\PSU_6(2)\le G \le
\Aut(\PSU_6(2))$ are characterised by such information. Indeed in such groups, the centralizer of a $3$-element has a similar structure to that in $\F_4(2)$ or $\Aut(\F_4(2))$ but in these groups $Z(Q)$ is weakly closed in $Q$, while in $\F_4(2)$ and its automorphism group it is not. (Recall that, for subgroups $X \le  Y \le L$, we say that $X$ is \emph{weakly closed} in $Y$ with respect
to $L$ provided that if $ g \in  L$ and $X^g \le Y$, then $X^
g = X$.)  Unfortunately the arguments in these different situations are quite different.  The theorems proved in \cite{PS1} and in this article  are   employed in
\cite{PS2} to identify the corresponding groups.

We now  make precise what we  mean by the approximate structure of the centralizer of an element of order 3 in $\PSU_6(2)$ or $\F_4(2)$.

\begin{definition} We say that  $X$ is similar to a $3$-centralizer in a group of type  $\PSU_6(2)$ or $\F_4(2)$
provided the following conditions hold.
\begin{enumerate}
\item $Q=F^*(X)$ is extraspecial of order $3^5$ and $Z(F^*(X)) =Z(X)$; and
\item $X/Q$ contains a normal subgroup isomorphic to $ \Q_8\times \Q_8$.
\end{enumerate}
\end{definition}

Our main theorem is as follows.

\begin{theorem}\label{MT} Suppose that $G$ is a group, $Z \le G$ has order $3$.
If $C_G(Z)$ is similar to a $3$-centralizer in a  group of type $\PSU_6(2)$ or $\F_4(2)$ and $Z$ is not weakly
closed in $F^*(C_G(Z))$, then  $G \cong \F_4(2)$ or $\Aut(\F_4(2))$.
\end{theorem}

Combining Theorem~\ref{MT}  and the main theorem from \cite{PS1} we obtain the following statement.

\begin{theorem}\label{combT} Suppose that $G$ is a group, $Z \le G$ has order $3$.
If $C_G(Z)$ is similar to a $3$-centralizer in  a  group of type $\PSU_6(2)$ or $\F_4(2)$ and $Z$ is not weakly
closed in a Sylow $3$-subgroup of $C_G(Z)$  with respect to $G$, then  either $F^*(G) \cong \F_4(2)$ or  $F^*(G) \cong
\PSU_6(2)$.
\end{theorem}

For groups $G$ with $C_G(Z)$ of type $\PSU_6(2)$ or $\F_4(2)$, the different $G$-fusion of $Z$ in $C_G(Z)$  manifests itself in the subgroup structure of $G$ very quickly.  Indeed, if we let $S$ be a Sylow $3$-subgroup of $C_G(Z)$ and $Q= F^*(C_G(Z))$, then we easily
determine that  $S \in \syl_3(G)$ and the Thompson subgroup $J$ of $S$ has order $3^4$ or $3^5$ when $Z$ is weakly closed in
$Q$,   whereas, it has order $3^4$ if $Z$ is not weakly closed in $Q$.  More strikingly, setting $L=
N_G(J)$, we have $F^*(L/Q) \cong \Omega_4^-(3)$ in the first case and in the second case $L/Q
\cong\Omega_4^+(3)$.

The paper is set out as follows. In Section~2 we gather pertinent information about that natural and spin modules
for $\Sp_6(2)$ and the natural and orthogonal $\SU_4(2)$-module as well as collect together further
identification theorems and results which we shall require for the proof of Theorem~\ref{MT}. In Section~3 we present  Theorem~\ref{P=F4} which will be used to identify a subgroup $P$  of our target group which is isomorphic to $\F_4(2)$. The proof of Theorem~\ref{P=F4} involves the construction of a building of type $\F_4(2)$ on which $P$ acts faithfully. The proof of the main theorem commences in Section~4. Thus we assume that $G$ satisfies the hypothesis of Theorem~\ref{MT} and set
$M= N_G(Z)$. We remark here that the information that is developed as the proof of Theorem~\ref{MT} unfolds
becomes information about the groups $\F_4(2)$ and $\Aut(\F_4(2))$ once the theorem is proved. The initial
objective of Section~4 is to determine more information about the structure of $M$. This is achieved by
exploiting the fact that $Z$ is not weakly closed in $Q=O_3(M)$. The first significant  result is presented in
Lemma~\ref{structM} where it is shown that $$M/Q \approx (\Q_8\times \Q_8).\Sym(3) \text{ or } (\Q_8\times
\Q_8).(2\times\Sym(3)).$$ In Section~4, we then move on, in Lemma~\ref{NJ}, to the determination of $L$ as
described in the previous paragraph.  At this stage we have  shown that $L \approx 3^4:\GO_4^+(3)$ or
$3^4:\mathrm{CO}_4^+(3)$. Thus $J$ supports a quadratic form and $G$-fusion of elements in $J$ is controlled by $L$.
This allows us to parameterize the  non-trivial cyclic subgroups of $J$ as singular, plus and minus (the
latter two types are fused when $L\approx 3^4:\mathrm{CO}_4^+(3)$) and also the five types of subgroups of order
$9$ which we label Type S, Type DP, Type DM, Type N+  and  Type N- (the notation is chosen to indicate that the groups are singular, degenerate with three plus groups, degenerate with three minus groups, non-degenerate of plus-type and non-degenerate of minus-type).

We let $\rho_1$ and $\rho_2$ be elements of $Q \cap J$ each centralized by a $\Q_8$ (the quaternion group of order $8$) subgroup of $M$ and one
generating a plus type and the other a minus type cyclic subgroup of $J$. In Section~6, we show that $C_G(\rho_1)
\cong C_G(\rho_2) \cong 3 \times \SU_4(2)$ or $3 \times \Sp_6(2)$ see Lemmas~\ref{eitheror} and \ref{thesame}. It
is the latter possibility that actually arises in our target groups. There is  related work  in \cite{FF} that we
might refer to at this stage but  they assume that $G$ is of characteristic $2$-type.

We let $r_1$ and $r_2$ be central involutions in the subgroup of $C_G(Z)$ isomorphic to $\Q_8\times \Q_8$ which
do not invert $Q/Z$ and, for $i=1,2$, we set $K_i= C_G(r_i)$. Again when $L\approx \mathrm{CO}_4^+(3)$ these
groups are conjugate. At this stage we know that $r_i$  centralizes  the (simple) component of $C_G(\rho_i)$. The
heart of the proof of Theorem~\ref{MT} is contained in Sections~7, 8, 9 and 10 where we determine the structure of
$K_i$. Thus the aim is to show that $K_1$ and $K_2$ have shape $2^{1+6+8}.\Sp_6(2)$ where $O_2(K_1)$ and $O_2(K_2)$ are commuting products of an extraspecial group of order $2^9$ and an elementary abelian group of order $2^7$.

We  begin our construction of $K_i$ by determining a large $2$-group $\Sigma_i$ which is normalized by $I_i=
C_J(r_i)$. It turns out that $\Sigma_i$ is the  extraspecial $2$-group of order $2^9$ and plus type we are seeking. In the case
that $C_G(\rho_i) \cong 3 \times \SU_4(2)$, we are able to show that in fact $K_i = N_G(\Sigma_i) $ and
$N_G(\Sigma_i)/\Sigma_i \cong \Aut(\SU_4(2))$ or $\Sp_6(2)$ and this leads to a contradiction as explained in
Lemma~\ref{ItsSp62}. Thus we enter Section~9 knowing that $C_G(\rho_1) \cong C_G(\rho_2) \cong 3 \times
\Sp_6(2)$. On the other hand $\Sigma_i$ is far from being a maximal signalizer for $I_i$. Thus is Section~9 we
construct an even larger signalizer which in the end is a product  $\Gamma_i=\Sigma_i\Upsilon_i$ where
$\Upsilon_i$ is an elementary abelian group of order $2^7$. Thus $\Gamma_i$ has order $2^{15}$ and in fact
$\Upsilon_i= Z(\Gamma_i)$ and this is proved in Lemma~\ref{Gammabasic}. We show that $N_{G}(\Gamma_i)/\Gamma_i
\cong \Sp_6(2)$ in Lemma~\ref{itssp}. The final hurdle requires that we show that $K_i=N_G(\Gamma_i)$. This is
proved in Lemma~\ref{H=K} and requires a sequence of lemmas which begins by showing that $\Upsilon_i$ is strongly
closed in $\Gamma_i$ with respect to $K_i$ and culminates in the statement that $\Upsilon_i$ is strongly closed
in a Sylow $2$-subgroup of $K_i$ with respect to $K_i$. At this stage we apply a Lemma~\ref{Gold} which is
essentially Goldschmidt's Strongly Closed Abelian $2$-subgroup Theorem \cite{Goldschmidt} to conclude that $K_i=
N_G(K_i) \approx 2^{1+6+8}.\Sp_6(2)$. Our final section exploits Theorem~\ref{P=F4}  to  produce a subgroup  $P$ of $G$ with $P \cong \F_4(2)$. We show that a
group closely related to $P$ is strongly $3$-embedded in $G$ and finally apply Holt's Theorem \cite{Ho} in the
form presented in Lemma~\ref{Holt} to conclude the proof of the Theorem~\ref{MT}.

 Throughout this article we follow the now standard Atlas \cite{Atlas} notation for
group extensions. Thus $X\udot Y$ denotes a non-split extension of $X$ by $Y$, $X{:}Y$ is a split extension of
$X$ by $Y$ and we reserve the notation $X.Y$ to denote an extension of undesignated type (so it is either
unknown, or we don't care). Our notation follows that in \cite{AschbacherFG}, \cite{Gorenstein} and  \cite{GLS2}.
We use the definition of signalizers as given in \cite[Definition 23.1]{GLS2}. For odd primes $p$, the extraspecial groups of
exponent $p$ and order $p^{2n+1}$ are denoted by $p^{1+2n}_+$. The extraspecial $2$-groups of order $2^{2n+1}$
are denoted by $2^{1+2n}_+$ if the maximal elementary abelian subgroups have order $2^{1+n}$ and otherwise we
write $2^{1+2n}_-$. We expect our notation for specific groups is self-explanatory. For a subset $X$ of a group
$G$, $X^G$ denotes the set of $G$-conjugates of $X$. If $x, y \in H \le  G$, we  write $x\sim _Hy$ to indicate
that $x$ and $y$ are conjugate in $H$. Often we shall give suggestive descriptions of groups which indicate the
isomorphism type of certain composition factors. We refer to such descriptions as the \emph{shape} of a group.
Groups of the same shape have normal series with isomorphic sections.  We use the symbol $\approx$ to indicate
the shape of a group.

\medskip

\noindent {\bf Acknowledgement.}  The first author is  grateful to the DFG for their support and thanks the mathematics department in Halle for their generous hospitality  from January to August 2011.

\section{Preliminaries}

In this section we lay out certain facts about the groups $\Sp_6(2)$ and $\Aut(\U_4(2))$ which play a pivotal role in the proof of our main theorem. We also present other background results that are of key importance to our investigations.

\begin{lemma}\label{modfacts} Suppose that $X \cong \Sp_6(2)$ or  $\Aut(\SU_4(2))$. Then there is a  unique irreducible $\GF(2)X$-module of dimension $6$ and a unique irreducible $\GF(2)X$-module of dimension 8 all the other non-trivial irreducible $\GF(2)X$-modules have dimension at least $9$.
\end{lemma}

\begin{proof} This is well known. See \cite{MOAT}. \end{proof}

In this section $U$ will denote  the  $\Aut(\SU_4(2))$ natural module  and the  $\Sp_6(2)$ spin module of dimension $8$ and $V$ will be the
$\Aut(\SU_4(2))$ orthogonal module and the $\Sp_6(2)$ natural module of dimension $6$.

\begin{table}
{\tiny
\begin{tabular}{|c|c|c|c|c|c|}
\hline
&& Centralizer in $\Aut(\SU_4(2))$ & Centralizer in $\Sp_6(2)$&$\dim C_U(u_j)$&$\dim C_V(u_j)$\\
\hline $a_2$&$u_1$&$2^{1+4}_+.(\SL_2(2) \times \SL_2(2))$&$2^{1+2+4}.(\SL_2(2) \times \SL_2(2))$&6&4\\
$b_3$&$u_2$&  $2 \times (\Sym(4) \times 2)$ &$2^7. 3$&4&3\\
$b_1$&$u_3$&$2 \times \Sp_4(2)$&$2^5.\Sp_4(2)$&4&5\\
$c_2$&$u_4$& $2^6. 3$ &$2^8. \SL_2(2)$&4&4\\\hline
\end{tabular}
\caption{Involutions in $\Sp_6(2)$ and $\Aut(\SU_4(2))$. The involutions in the first row are the \emph{unitary
transvections.} The involutions labeled with ``$b$"  those which are in $\Aut(\SU_4(2)) \setminus \SU_4(2)$.}
\label{Table1}}
\end{table}

For $X \cong \Sp_6(2)$,  let $X_1, X_2 $ and $X_3$  be the minimal parabolic subgroups of $X$ containing a fixed Sylow $2$-subgroup $S$. Set $X_{ij}= \langle X_i, X_j\rangle$ where $1 \le i<j\le 3$ and fix notation so that $$X_{12}/O_2(X_{12})\cong \SL_3(2),$$ $$X_{23}/O_2(X_{23}) \cong \Sp_4(2) \text{ and} $$ $$X_{13}/O_2(X_{13}) \cong \SL_2(2)\times \SL_2(2).$$
There are three conjugacy classes of elements of order $3$ in $X$. Let $\tau_1$, $\tau_2$ and $\tau_3$
be representatives of these classes and choose  so that on  the natural $\Sp_6(2)$-module $V$, for $1
\le i \le 3$,  $\dim [V,\tau_i]= 2i$.

\begin{lemma}\label{sp62facts} Suppose that $Y \cong \Aut(\SU_4(2))$ and that $X \cong \Sp_6(2)$ with $Y \le X$.
Assume that $V$ and $U$ are the faithful $\GF(2)X$-modules of dimension $6$ and $8$ respectively.
\begin{enumerate}
\item $X$ and $Y$ each have four conjugacy classes of involutions and for each involution $u\in X$ we have
$u^X\cap Y$ is a conjugacy class in $Y$. In column one of Table~\ref{Table1}  we { provide the  Suzuki names (see \cite[page 16]{AschSe}) for each class of involutions.}
 \item The shape of the centralizers of involutions in $X$ and $Y$ is
given in Table~\ref{Table1}.
\item For each involution in $u\in X$, $\dim C_V(u)$ and $\dim C_U(u)$ is given in Table~\ref{Table1}.
\item $X$ does not contain any subgroup of order $2^4$ in which all the involutions are conjugate.
\item  $X$ does not contain an extraspecial subgroup of order $2^7$.
    \item If $x$ is an involution of type $b_1$, then  a Sylow $3$-subgroup of $C_Y(u)$ contains two conjugates of $\langle \tau_1\rangle $ and two  conjugates of $\langle \tau_2\rangle$.
    \item $E=\langle \tau_1,\tau_2,\tau_3\rangle$ is the Thompson subgroup of a Sylow $3$-subgroup of $G$ and every element of order $3$ is $X$-conjugate ($Y$-conjugate) to an element of $E$.
\end{enumerate}
\end{lemma}

\begin{proof} {Parts (i)-(iii) follow from \cite[Proposition 2.12, and Table 1]{PS1}.}

Suppose that $A \le X$ has order $2^4$ and that all the non-trivial elements are conjugate in $X$. We use the character table of $X$ given in \cite[page 47]{Atlas}.  Let $\chi$ be an irreducible character of $X$. Then, as $(\chi|_A,1_A) \ge 0$, we  have $$(\chi|_A,1_A) = \frac{1}{|A|} \sum_{a\in A} \chi(a) \ge 0.$$ Taking $\chi$ to be the degree $7$ character we see that all the non-trivial elements in $A$ are  in Suzuki  class $c_2$ (Atlas \cite{Atlas} $2C$). Now considering the character of degree $35$  denoted $\chi_7$ in \cite{Atlas} we obtain a contradiction.

Let $E$ be extraspecial of order $2^7$.
Since $X$ has a faithful
$7$-dimensional representation in characteristic $0$ and the smallest such representation of $E$ is $8$-dimensional, $E$ is not isomorphic to a subgroup of $X$.

Part (vi) follows from the action of $\Sp_4(2)$ on the natural module for $\Sp_6(2)$ as $\Sp_4(2)$ contains no conjugates of $\tau_3$.

Part (vii) is also  elementary  to verify.

\end{proof}

\begin{lemma}\label{sp62natural} Let $X \cong \Sp_6(2)$,  $S$ a Sylow $2$-subgroup of $X$ and $V$ be the $\Sp_6(2)$ natural module.
Then the following hold.
\begin{enumerate}
\item $X$ acts transitively on the non-zero vectors in $V$.
\item $V$ is uniserial as an $S$-module.
\item Suppose that, for $1 \le i \le 3$, $V_i$ is an $S$-invariant subspace of $V$ of dimension $i$. Then $X_{23}=N_X(V_1)$ and $X_{23}$ acts naturally as $\Sp_4(2)$ on $V_1^\perp/V_1$, $X_{13}=N_X(V_2)$, $O^2(X_3)$ centralizes $V_2$ and $V/V_2^\perp$, and $O^2(X_1)$ centralizes $V_2^\perp/V_2$ and $X_{12}=N_X(V_3)$ and acts naturally on both $V_3$ and $V/V_3$.
\end{enumerate}
\end{lemma}

\begin{proof} These are all well known facts about the action of $X$ on $V$. See for example \cite[Lemma 14.37]{SymplecticAmalgams} for (i) and (ii).
\end{proof}

\begin{lemma}\label{sp62spin}Let $X \cong \Sp_6(2)$, $S$ a Sylow $2$-subgroup of $X$ and $U$ be the  $\Sp_6(2)$ spin module.
\begin{enumerate}
\item $X$ has exactly two orbits on the non-zero vectors of $U$ one of length $135$ and one of length $120$.
\item $N_X(C_U(S)) = X_{12}$ and $C_U(S)= C_U(O_2(X_{12}))$.
\item If $U_2\le U$ is $S$-invariant of dimension $2$, then $N_X(U_2)=X_{13}$ and $O^2(X_1)$ centralizes $U_2$.
\end{enumerate}
\end{lemma}

\begin{proof}  See \cite[Proposition 2.12]{PS1}.
\end{proof}

\begin{lemma}\label{sp62line} Suppose that $X \cong \Sp_6(2)$ and
$V$ is the natural module for $X$. Let $P= X_{13}$,  $T \in \syl_3(P)$ and $Q= O_2(P)$.
\begin{enumerate}
\item $P/Q \cong \SL_2(2) \times \SL_2(2)$.
\item  The subgroups of order $3$ in $T$ are as follows:  there are  two subgroups $Z_1$ and $Z_2$ which are $X$-conjugate to $\langle \tau_3\rangle$, one subgroup which is $X$-conjugate to $\langle \tau_1\rangle$ (which we suppose is $\langle \tau_1\rangle$) and one subgroup which is $X$-conjugate to $\langle \tau_2\rangle $. The two  subgroups of $T$ which are conjugate to $\langle \tau_3\rangle$ are conjugate in $N_P(T)$.
\item $C_Q(Z_1) \cong C_Q(Z_2) \cong \Q_8$ and $[C_Q(Z_1),C_Q(Z_2)]=1$.
\item $C_T(Z(Q))=\langle \tau_1\rangle$ and $C_Q(\tau_1)= Z(Q)$.
\item If $U\le Q$ has order $2^3$ and if $U$ is $T$-invariant, then either $U = C_Q(Z_1)$,  $U= C_Q(Z_2)$ or $U= Z(Q)$.
    \item Let $Q^\prime = \langle t \rangle$. Then $t^X \cap Q\not\subseteq Z(Q)$.
\end{enumerate}
\end{lemma}

\begin{proof} Let $Y$ be the $P$-invariant  isotropic $2$-space in $V$. Then $P$ preserves $0<Y < Y^\perp <V$.
Let $I$ be a hyperbolic line and $J= I^\perp$ be chosen so  $Y \le J$. Then the decomposition $I \perp J$ is
preserved by $\Sp_2(2) \times \Sp_4(2)$ and the subgroup $K$ of this group which leaves  $Y$ invariant has shape
$\Sp_2(2) \times (2\times 2^2).\SL_2(2) \cong \SL_2(2) \times 2 \times \Sym(4)$.  In particular, we now have
(i) holds. Furthermore,  we may suppose  the first factor of $K$ contains $\langle \tau_1\rangle$ while the
second factor contains $\langle \tau_2^*\rangle$, an $X$-conjugate of $\langle \tau_2\rangle$, acting fixed point
freely on $J$. Set $T =\langle \tau_1, \tau_2^*\rangle$. Since $\tau_1$ is inverted in the first factor of $K$,
we see  the two diagonal products $\tau_1\tau_2^*$ and $\tau_1^2\tau_2^*$ are conjugate in $N_P(T)$.
Furthermore these elements act fixed point freely on $V$ and so are $X$-conjugate to $\tau_3$. This is (ii).

Now consider $Q$. We know  this group has order $2^7$. We further have $Q \cap K = O_2(K)$ centralizes $Y+I=
Y^\perp$. Consequently $Q\cap K$ is normal in $P$ and as $[V,Q,Q\cap K]= [V,Q\cap K , Q]$ we additionally have $K\cap Q
\le Z(Q)$. Note that $\langle \tau_1\rangle$ centralizes $Q\cap K$. Now $C_P(\tau_2^*)$  is contained in
$K$ and so we see  $C_Q(\tau_2^*) = Z(K)$ has order $2$.  Now the centralizer in $X$ of $\tau_3$ supports a
$\GF(4)$ structure and is isomorphic to $\SU_3(2)$. It follows that  $\tau_1\tau_2^*$ and $\tau_1^2\tau_2^*$ can
centralize only quaternion subgroups of order $8$ in  $Q$. Since $C_Q(\tau_1\tau_2^*)$ and
$C_Q(\tau_1^2\tau_2^*)$ both centralize $Z(K)$ and $|Q|=2^7$ we have  $C_Q(\tau_1\tau_2^*) \cong
C_Q(\tau_1^2\tau_2^*) \cong \Q_8$ and  $C_Q(\tau_1\tau_2^*)'= Z(K)$.  Putting $Q_1 =
C_Q(\tau_1\tau_2^*)C_Q(\tau_1^2\tau_2^*)$ we have  $Q_1$ is $T$-invariant. Now $Q=
C_Q(\tau_1\tau_2^*)C_Q(\tau_1^2\tau_2^*)(Q\cap K)$,
$$[Q,\tau_1]= [C_Q(\tau_1\tau_2^*),\tau_1][C_Q(\tau_1^2\tau_2^*), \tau_1] = Q_1$$ is a normal subgroup of $Q$
and $Q_1\cap (Q\cap K)\le Z(K)$. Thus $Q_1$ is extraspecial and $Q'=Z(K)$ which has order $2$. In addition, $Q=
C_Q(\tau_1\tau_2^*) [Q,\tau_1\tau_2^*]$ with $C_Q(\tau_1\tau_2^*) \cap [Q,\tau_1\tau_2^*]= Z(K)$. Since
$$[C_Q(\tau_1\tau_2^*) , Q,\tau_1\tau_2^*] \le [Z(K),\tau_1\tau_2^*]=1$$ and $[C_Q(\tau_1\tau_2^*)
,\tau_1\tau_2^*,Q] =1,$ we also have $[C_Q(\tau_1\tau_2^*) , [Q,\tau_1\tau_2^*]]=1$ by the Three Subgroup Lemma.
In particular, as $[Q,\tau_1\tau_2^*]= C_Q(\tau_1^2\tau_2^*)(Q \cap K)$, we now have  (iii) and (iv) hold. If
$U$ is of order $2^3$ and is $T$-invariant, then $C_T(U)> 1$ and so (v) also follows from the above discussion.
To prove (vi), we start with a transvection $r \in Z(Q)$. By Table~\ref{Table1} we have  $E = O_2(C_X(r))$ is
elementary abelian of order $2^5$. Now $|E \cap Q| \geq 2^3$. If $E \cap Q \leq Z(Q)$, then, as $E \leq
C_{N_X(Q)}(E \cap Q)$, we get  $|E \cap Q| \geq 2^4$, a contradiction. Hence $E \cap Q \not\leq Z(Q)$. Now as
$N_X(E)$ acts transitively on $E/\langle r \rangle$, we have  any coset of $\langle r \rangle$ in $E$
contains a conjugate of $t$. In particular $t^X \cap E \cap Q \not\subseteq Z(Q)$.
\end{proof}

 \begin{lemma}\label{Noover} Let $Y = \Aut(\SU_4(2))$ and $V$ be the natural $\OO^-_6(2)$-module. Then there are no elementary abelian subgroup $E$ of order $8$ in $Y$ such that $|V : C_V(E)| \leq 4$.
\end{lemma}

\begin{proof}
Suppose false and let $E$ be such a subgroup of order $8$. From Table~\ref{Table1} we see  $E$ cannot contain elements of type $b_3$. If $E \not\leq
Y^\prime$, then $E$ contains exactly four elements of type $b_1$. As there are at most three hyperplanes in $V$
containing $C_V(E)$, two of these  elements have to centralize the same hyperplane of $V$. But then their
product, which is an involution in $E \cap Y$, also centralizes this hyperplane. As $\Omega_6^-(2)$ does
not contain transvections, we have $E \leq Y^\prime$. Therefore $|V : C_V(E)| = 4$ and $C_V(E) = C_V(e)$ for
all $e \in E^\#$. As $C_V(e) = [V,e]^\perp$ we also have $[V,e] = [V,E]$ for all $e \in E^\#$ which means all the
involutions in $E$ are conjugate. Now we use the character table of $\SU_4(2)$ as in the proof of
Lemma~\ref{sp62facts}(iv) to obtain a contradiction.
\end{proof}
%

Recall that a faithful $\GF(p)G$-module is an \emph{$F$-module} provided there exists a non-trivial elementary abelian $p$-subgroup $A \le G$ such that $|V:C_V(A)|\le |A|$. The subgroups $A \le G$ with $|V:C_V(A)|\le |A|$ are called \emph{offenders}. 

\begin{lemma} \label{NotF}Suppose that $X \cong \Sp_6(2)$ or $\Aut(\SU_4(2))$ and $W$ is a $\GF(2)X$-module of dimension $14$ which has exactly two composition factors one of dimension 6 and one of dimension $8$. Then $W$ is not an $F$-module.
\end{lemma}

\begin{proof} Suppose that $A \le X$ is an  offender on $W$. Then $|A| \ge |W:C_W(A)|$. From Table~\ref{Table1},  for $a \in A$, we read $|A| \ge |W:C_W(a)| \ge 2^4$. Since the $2$-rank of $X$ is at most $6$, we  also have that $A$ does not contain any involutions of type $b_3$.

Suppose that $|A|=2^4$. Then all the involutions in  $A$ must be of type  $a_2$. This contradicts Lemma
\ref{sp62facts}(iv).  Hence $|A| \ge  2^5$  and $X \cong \Sp_6(2)$  as the $2$-rank of  $\Aut(\SU_4(2))$ is $4$
(see \cite[Proposition 2.12 (x)]{PS1}). We use the notation for involutions from Table~\ref{Table1}. We may as well suppose  $A \le C_X(u_3)$. Then as the $2$-rank of
$\Sp_4(2)$ is $3$, we have  $A\cap O_2(C_X(u_3))\not=1$. Since $|C_U(O_2(C_X(u_3)))|= 2^4$ and
$|C_V(O_2(C_X(u_3)))|=2$ certainly $A \not = O_2(C_X(u_3))$.  Now $O_2(C_X(u_3))$ contains 15 elements from
$u_1^X$, 15 elements from $u_4^X$ and one element from $u_3^X$ and multiplication by $u_3$ maps $u_1^X \cap
O_2(C_G(u_3))$ to $u_4^X \cap O_2(C_X(u_3))$. Thus, if $A$ contains a conjugate of $u_3$, then $A \cap
u_i^X \not=\emptyset $ for $i=1,3,4$. As $|A|= 2^5$, $A$ does not consist purely of elements of elements from
class $u_1^X$ by Lemma~\ref{sp62facts} (iv)  and consequently we must have elements from $u_4^X$ in $X$. It
follows now from Table~\ref{Table1} that $|A| = 2^6$. There is a unique such elementary abelian subgroup in a
Sylow $2$-subgroup of $X$ and its normalizer is a plane stabiliser in the action  of $X$ on $V$. But then
$|W:C_W(A)|\ge 2^{10}$ which is a contradiction.
\end{proof}

\begin{lemma}\label{nonsplitmods} Suppose that $X \cong \Sp_6(2)$, $W$ is a 7-dimensional
$\GF(2)X$-module with  $W/C_W(X)$  the natural  $\Sp_6(2)$-module. If $S \in \syl_2(X)$, then $C_W(S)> C_W(X)$.
\end{lemma}

\begin{proof}  Consider the subgroup $K= K_1\times K_2$ of $X$
which preserves the decomposition of $W/C_W(X)$ in to a perpendicular sum of a non-degenerate $2$-space $A/C_W(X)$  and
a non-degenerate $4$-space $B/C_W(X)$ with $K_1\cong \Sp_2(2)$ and $K_2 \cong \Sp_4(2)$.   Let $t$ be an
involution in $K_1$. Since $\dim A=3$, we have $\dim [A,t]=1$. Furthermore  $B/C_B(t)\cong [B,t]$ as
$K_2$-modules and so we must have $[B,t]=0$. Thus $[W,t] = [A,t]+[B,t]= [A,t]$ has dimension 1 and so $t$ is a
transvection on $W$.  Let $P$ be
the stabiliser in $X$ of the $2$-space  $L=[A,t]+C_W(X)$. Then $P= C_X(t)$ and contains $K_2$ and a Sylow $2$-subgroup $S$ of
$X$. Furthermore,  $P' = O_2(P)K_2'$ has index $2$ in $P$. Since $P'$ is perfect it centralizes $L$. Let
$s \in P\setminus P'$ be a transvection on $V$. Then $s$ is $X$-conjugate to $t$, and so $s$ is a transvection on $W$.   Then $[L,s]\le C_W(X)$ and if $[L,s]= C_W(X)$ then certainly $[W,s]=
C_W(X)$ which is nonsense. Thus $[L,s]=0$ and we conclude that $P$ and hence $S$ centralizes $L$. This proves the
claim.
\end{proof}

\begin{theorem}[Prince] \label{PrinceThm}
Suppose that $Y$ is isomorphic to the centralizer of a $3$-central element of order $3$ in $\PSp_4(3)$ and that $X$
is a finite group with a non-trivial element $d$ such that $C_X(d)\cong Y$. Let $P \in \Syl_3(C_X(d))$ and $E$ be
the elementary abelian subgroup of $P$ of order $27$. If $E$ does not normalize any non-trivial
$3^\prime$-subgroup of $X$ and $d$ is $X$-conjugate to its inverse, then either
\begin{enumerate}
\item
$|X:C_X(d)| =2$;
\item $X$ is isomorphic to $\Aut(\SU_4(2))$; or
\item $X$ is isomorphic to $\Sp_6(2)$.
\end{enumerate}
\end{theorem}

\begin{proof} See \cite[Theorem 2]{prince1}. \end{proof}

\begin{lemma} \label{cen3psp43}Suppose that $X$ is a group of shape $3^{1+2}_+.\SL_2(3)$,  $O_2(X)=1$
 and
a Sylow $3$-subgroup of $X$ contains an elementary abelian subgroup of order $3^3$. Then $X$ is isomorphic to the
centralizer of a non-trivial $3$-central element in $\PSp_4(3)$.
\end{lemma}
\begin{proof} See \cite[Lemma~6]{Parker1}.\end{proof}

\begin{lemma}\label{quadratic form}
Suppose that $F$ is a field, $V$ is an $n$-dimensional vector space over $F$ and  $G= \GL(V)$. Assume that $q$ is quadratic form  of Witt index at least $1$ and with non-degenerate associated bilinear form $f$, where, for $v,w \in V$, $f(v,w) = q(v+w)-q(v)- q(w)$.
Let $\mathcal S$ be the set of singular 1-dimensional subspaces of $V$ with respect to $q$. Then the stabiliser in $G$ of $\mathcal S$ preserves $q$ up to similarity.
\end{lemma}

\begin{proof} See \cite[Lemma 2.10]{ParkerRowley}.
\end{proof}

\begin{lemma}\label{GO4} Suppose that $p$ is an odd prime, $X = \GL_4(p)$ and $V$ is the natural $\GF(p)G$-module. Let $A =\langle a, b\rangle\le X$ be elementary abelian of order $p^2$ and assume that $[V,a] = C_V(b)$ and $[V,b]= C_V(a)$ are distinct and of dimension $2$.
Let $v \in V\setminus [V,A]$. Then  $A$ leaves invariant a non-degenerate quadratic form with respect to which  $v$ is a singular vector. In particular, $X$ contains exactly two conjugacy classes of subgroups such as $A$. One is conjugate to a Sylow $p$-subgroup of $\GO_4^+(p)$ and the other to a Sylow $p$-subgroup of $\GO_4^-(p)$.
\end{lemma}
\begin{proof} See \cite[Lemma 2.11]{ParkerRowley}.
\end{proof}

The $4$-dimensional orthogonal module of $+$-type will play a prominent role in the proof of our main theorem. We next introduce some notation which will be used in the proof.

\begin{notation}\label{o4} Let $V$ be a  $4$-dimensional non-degenerate orthogonal space of $+$-type over $\GF(3)$. Assume that $X$ is a non-zero subspace of $V$.
Then $\mathcal S(X)$ is the set of singular 1-dimensional subspaces in $X$, $\mathcal P(X)$ the set of 1-dimensional subspaces of $+$-type in $X$ and $\mathcal M(X)$ the set of 1-dimensional subspaces of $-$-type in $X$.
\end{notation}

\begin{lemma}\label{hyper}  Let $X$ be a $3$-dimensional subspace  in a non-degenerate 4-dimensional orthogonal space of $+$-type over $\GF(3)$.  Then $\mathcal S(X) \not=\emptyset$.
\end{lemma}

\begin{proof} See \cite[21.3]{AschbacherFG}.
\end{proof}

We now introduce some additional  notation:

\begin{notation}\label{type} Let $V$ be a $4$-dimensional non-degenerate orthogonal space of $+$-type over $\GF(3)$  and $E$ be a $2$-dimensional subspace of $V$. The type of $E$ is determined by the number of one spaces of a given type in $E$. Thus we have

\begin{tabular}{ll}
{\rm Type S}:&$|\mathcal S(E)|=4$.\\
{\rm Type DP}:& $|\mathcal S(E)|=1$ and $|\mathcal P(E)|=3$.\\
{\rm Type DM}:&$|\mathcal S(E)|=1$ and $|\mathcal M(E)|=3$.\\
{\rm Type N+}:&$|\mathcal S(E)|=2$ and  $|\mathcal M(E)|=|\mathcal P(E)|=1$.\\
{\rm Type N-}:& $|\mathcal P(E)|=|\mathcal M(E)|=2$.
\end{tabular}
\end{notation}

\begin{lemma}\label{types} Let $V$ be a $4$-dimensional non-degenerate orthogonal space over $\GF(3)$ of $+$-type and $E$ be a $2$-dimensional subspace of $V$. Then $E$ is of one of the types in Notation~{\rm \ref{type}}.
\end{lemma}

\begin{proof} The subspaces of $V$ of dimension $2$ are either totally singular (S), degenerate with three elements of $\mathcal P(V)$ (DP), degenerate with three elements from $\mathcal M(V)$ (DM) , non-degenerate of plus type (N+),
 or non-degenerate of minus type (N-).
\end{proof}

\begin{theorem}\label{closed} Suppose that $G$ is a   finite group, $Q$ is a
subgroup of $G$ and $H= N_G(Q)$. Assume that the following  hold
\begin{enumerate}
\item $H/Q \cong \Aut(\SU_4(2))$ or $\Sp_6(2)$;
\item $Q=C_G(Q)$ is a minimal
normal subgroup of $H$ and is elementary abelian of order $2^8$;
\item $H$ controls $G$-fusion of elements of $H$ of order $3$; and
\item if $g \in G\setminus H$ and $d \in H\cap H^g$ has order $3$,
then $C_Q(d)=1$.
\end{enumerate}
Then $G= HO_{2'}(G)$.
\end{theorem}

\begin{proof} This is \cite[Theorem 3.1]{ParkerRowley}.
\end{proof}

\begin{lemma}\label{involutionsonexspec}
Suppose that $G$ is a group, $E$ is an extraspecial $2$-group  which is normal in $G$ and $x \in G\setminus C_G(E)$ is
an involution. If $x $ is not $E$-conjugate to $xe$ where $e \in Z(E)^\#$, then $C_E(x)\ge [E,x]$ and $[E,x]$ is elementary
abelian.
\end{lemma}

\begin{proof}  Certainly  $C_{E/Z(E)}(x)\ge [E/Z(E),x]$. Therefore, if $C_E(x) \not \ge [E,x]$, then  $[f,x,x] = e$ for some $f \in E$. Setting $w = [f,x]$ we then have $x^w=xe$ which contradicts our hypothesis on $x$. Hence $C_E(x) \ge [E,x]$.

We now show that every element of $[E,x]$ has order $2$. Let $f \in [E,x]$. Then $fe$  has the same order as $f$. Thus we may suppose that $f=[h,x]$ for some $h \in E$.  As $[E,x]\le C_E(x)$,  $x[h,x]=[h,x]x$  and so 
\begin{eqnarray*}f^2&=&[h,x][h,x] = h^{-1}xhx[h,x]=h^{-1}xh[h,x]x\\&=& h^{-1}xhh^{-1}xhxx=1\\
 \end{eqnarray*}as required. This proves the lemma.
\end{proof}

For a group $X$ with subgroups $A \le Y \le X$,  we say that $A$ is \emph{strongly closed  in $Y$ with respect to
$X$} provided $A^x \cap Y \le A$ for all $x \in X$.

\begin{lemma}\label{Gold} Suppose that $K$ is a group, $O_{2'}(K)=1$,  $A$ is an abelian $2$-subgroup of $K$ and $A$ is strongly closed in $N_K(A)$ with respect to $K$. Assume that $F^*(N_K(A)/C_K(A))$ is a non-abelian simple group.
 Then  $K= N_K(A)$.
\end{lemma}

\begin{proof} Set  $ L= \langle A^{ K}\rangle$.   Since $O_{2'}(K)=1$, we have $O_{2'}(L)=1$.
By Goldschmidt \cite[Theorem A]{Goldschmidt},  $ L=  O_2( L)E( L)$ and $ A=  O_2( L)\Omega_1( T)$ where $ T \in \syl_2( L)$ contains $ A$. If $E( L)=1$, then $ A$ is normal in $ K$ and we are done.  Thus $E( L) \neq 1$.
 Goldschmidt additionally states that  $E( L)$ is a direct product of simple groups of type $\PSL_2(q)$, $q \equiv 3,5 \pmod 8$, ${}^2\G_2(3^a)$,  $\SL_2(2^a)$, $\PSU_3(2^a)$, ${}^2\B_2(2^a)$ for some natural number $a$, or the sporadic simple group $\J_1$.
It  follows from the structure of these groups that $N_{ L}( A)$ is a soluble group which is not a $2$-group.
On the other hand, $N_{ L}( A)=  L \cap {N_K(A)}$ is a normal subgroup of $ {N_K(A)}$.  Since $F^*(N_K(A)/C_K(A))$ is a non-abelian simple group and $N_{ L}( A)$ is soluble we now have $N_L(A) \le C_K(A)$ and this contradicts the structure of $E(L)$. Thus $A$ is normal in $K$ as claimed.
\end{proof}

 We will also need the following statement of Holt's Theorem \cite{Ho}.

 \begin{lemma}\label{Holt}  Suppose that $K$ is a simple group,  $P$ is a  proper subgroup of $K$ and  $r$ is a $2$-central element of $K$. If $r^K\cap P= r^P$ and $C_K(r)\le P$, then  $K\cong \PSL_2(2^a)$ ($a \ge 2$), $\PSU_3(2^a)$ ($a \ge 2$), ${}^2\B_2(2^a)$ ($a\ge 3$ and odd) or $\Alt(n)$ ($n \ge 5$) where  in the first three cases $P$ is a Borel subgroup of $K$ and in the last case $P \cong \Alt(n-1)$.
 \end{lemma}

\begin{proof}
Set $\Omega= K/P$ and assume that $P<K$. The conditions
$C_K(r) \le  P$ and $r^K \cap P= r^P$ together imply that $r$ fixes a unique point of $\Omega$. Let $J$ be the set of involutions of $K$ which fix exactly one point of $\Omega$.  Since $r$ is  a $2$-central element of  $K$, any $2$-group which fixes at least $3$ points when it acts on
$\Omega$ commutes with an element of $J$. Hence   Holt's criteria (*) from
\cite{Ho} is satisfied.
In addition, the simplicity of  $K$ yields $K= \langle r^K\rangle= \langle J\rangle$.   Thus \cite[Theorem 1]{Ho} implies that   $K$ is isomorphic to one of the following groups
$\PSL_2(2^n)$, $\PSU_3(2^n)$, ${}^2\B_2(2^n)$ ($n\ge 3$ and odd) or $\Alt(\Omega)$ where in the first three classes of
groups the stabiliser  $P$  is a Borel subgroup and in the latter case it is $\Alt(\Omega\setminus\{P\})$.
 \end{proof}

For the final steps in the identification of $\F_4(2)$  we need information about its involutions and their centralizers.

\begin{lemma}\label{F42Classes} The group $X=\F_4(2)$ has four conjugacy classes of involutions $x_1, x_2, x_3$ and $x_4$  three of which are $2$-central.
Furthermore we may assume that notation is chosen so that
\begin{enumerate}\item $C_X(x_1) \cong C_X(x_2) \approx 2^{1+6+8}.\Sp_6(2)$; \item
$C_X(x_3) \approx 2^{1+1+4+1+4+4+1+4}.\Sp_4(2)$; and \item $C_X(x_3) \approx 2^{[9]}.(\SL_2(2)\times \SL_2(2))$.\end{enumerate}
\end{lemma}

\begin{proof}
These facts can be found in Guterman \cite[Section 3]{Guterman}  (see also \cite[Page 45]{AschSe}) .
\end{proof}

\section{Identifying $\F_4(2)$}

The final step in the proof of Theorem~\ref{MT} demands  that we can identify $\F_4(2)$ or $\Aut(\F_4(2))$ from the structure of the centralizer of a certain $2$-central involution. In this section  we give such an identification. The centralizers of interest are the centralizers of the  involutions $x_1$, $x_2$ in $\F_4(2)$ as given in Lemma~\ref{F42Classes} (i). Of course, we do not want to specify the isomorphism type of such a centralizer, but only the approximate shape of the group.

\begin{definition}\label{F4cent}  We say  the group $U$ is similar to a $2$-centralizer in a group of type  $\F_4(2)$ if $U$ has the following properties. 
\begin{enumerate}
\item  $U/O_2(U) \cong \Sp_6(2)$,
\item  $O_2(U) $ is an product of $Z(O_2(U))$  by an extraspecial group of order $2^9$,  $Z(O_2(U))$ is elementary abelian of order $2^7$.
\item   $U/O_2(U)$ induces the natural module on $Z(O_2(U))/O_2(U)'$ and the spin module on $O_2(U)/Z(O_2(U))$.
\end{enumerate}
\end{definition}

\begin{definition}\label{F4setup} Suppose that $G$ is a group and assume  that the  following hold:
\begin{enumerate}
\item For $i=1,2$, there are  involutions $x_i$ in $G$ such that $U_i = C_G(x_i)$ is similar to a $2$-centralizer in a group of type  $\F_4(2)$.
\item There is a Sylow $2$-subgroup $T$   of $U_1$ such that  $Z(T) = \langle x_1,x_2 \rangle$. 
\end{enumerate}
Then we say that $U_1$, $U_2$, $T$ is an \emph{$\F_4$ set-up in $G$.}
\end{definition}

Our identification theorem in this section is as follows:
\begin{theorem}\label{P=F4} If $U_1$, $U_2$, $T$ is an $\F_4$ set-up in $G$, then  $\langle U_1, U_2 \rangle \cong \F_4(2)$.
\end{theorem}

For the remainder of this section we assume that $U_1$, $U_2$ and $T$ is an $\F_4$ set-up in $G$. Notice that because of Definition~\ref{F4cent} (ii),  for $i=1,2$,   $O_2(U_i)'= \langle x_i\rangle$ has order $2$. The first lemma details the relationship of $U_1$ with $U_2$.

\begin{lemma} \label{p23} The following hold:
 \begin{enumerate}
 \item $U_1\cap U_2 $ contains $T$;\item  $(U_1\cap U_2)/O_2(U_1\cap U_2)\cong \Sp_4(2)$; \item  $O_2(U_1\cap U_2)=O_2(U_1)O_2(U_2)$; and \item $Z(T) = Z(O_2(U_1)) \cap Z(O_2(U_2))$.\end{enumerate}
\end{lemma}

\begin{proof} From the definition of an $\F_4$ set-up in $G$, we have $T \le U_1 \cap U_2$. This proves (i).

Since $Z(U_i)/\langle x_i\rangle$ is a natural $U_i/O_2(U_i)$-module and $|Z(T)|=4$, Lemma~\ref{nonsplitmods} implies $Z(T) \le Z(U_1) \cap Z(U_2)$. Therefore, by Lemma~\ref{sp62natural} (iii),
\begin{eqnarray*}(U_1\cap U_2)/O_2(U_1\cap U_2)&=&C_{U_1}(Z(T))/O_2(C_{U_1}Z(T))\\ &=& C_{U_2}(Z(T))/O_2(C_{U_1}Z(T)) \cong \Sp_4(2).\end{eqnarray*}  Hence (ii) holds.

Since $$(O_2(U_1)
\cap O_2(U_2))' \le O_2(U_1)' \cap O_2(U_2)'= \langle x_1\rangle \cap \langle x_2\rangle=1,$$ $O_2(U_1)
\cap O_2(U_2)$ is abelian. Therefore,  as $O_2(U_1)$ contains an extraspecial subgroup of order $2^9$, we have $$|O_2(U_1): O_2(U_1)
\cap O_2(U_2)|\ge 2^4.$$  Furthermore, as $O_2(U_1)  O_2(U_2)/O_2(U_1) $ is normal in $(U_1\cap U_2)/O_2(U_1)$,   $O_2(U_1\cap U_2)=O_2(U_1)O_2(U_2)$ follows from Lemma~\ref{sp62natural} (iii). This is (iii).

Finally, since  $O_2(U_1\cap U_2)$ centralizes
$Z(O_2(U_1)) \cap Z(O_2(U_2))$, we deduce  $Z(T)= Z(O_2(U_1)) \cap Z(O_2(U_2))$ and this proves (iv).\end{proof}

Our method to prove Theorem~\ref{P=F4} is to use  the $\F_4$ set-up $U_1$, $U_2$, $T$ in $G$  to construct a chamber system of type $\F_4(2)$ using the subgroup  $P= \langle U_1, U_2\rangle$ of $G$. To accomplish this we first define  $P_1, P_2, P_3$ to be subgroups of $U_1$ containing $T$ such that $P_j/O_2(U_1)$, $j= 1,2,3$, are
the minimal parabolic subgroups  of $U_1/O_2(U_1)$ containing $T/O_2(U_1)$. We additionally let $P_4$ be such that
$U_2\ge P_4 \ge T$, $P_4 \not \le U_1$ and $P_4/O_2(U_2)$ is a minimal parabolic subgroup of $U_2/O_2(U_2)$.
For $\emptyset \neq \sigma \subseteq \{1,2,3,4\}$ we set $P_\sigma= \langle P_ j \mid j \in \sigma\rangle$.

  We may assume that notation has been chosen so that \begin{eqnarray*}P_{12}/O_2(P_{12})&\cong &\SL_3(2);\\
 P_{13}/O_2(P_{13})&\cong &\SL_2(2)\times \SL_2(2); \text{ and} \\
 P_{23}/O_2(P_{23})&\cong &\Sp_4(2).\\
 \end{eqnarray*}
 Note also that $P_j /O_2(P_j) \cong \SL_2(2)$ for $1 \le j \le 4$.
 By Lemma~\ref{p23} (ii), $P_{23}= U_1\cap U_2$ and $P= \langle P_1,P_2,P_3,P_4\rangle$.

 Set  $\mathcal I= \{1,2,3,4\}$,  and let $${\mathcal{C}} =
(P/T, (P/P_k), k \in \mathcal I)$$
\noindent  be the corresponding chamber system. Thus $\mathcal C$ is an edge coloured
graph  with colours from $\mathcal I=\{1,2,3,4\}$ and vertex set the right cosets $P/T$. Furthermore, two cosets
$Tg_1$ and $Tg_2$ form a $k$-coloured edge if and only if $Tg_2g_1^{-1} \subseteq P_k$. Obviously $P$ acts on
$\mathcal C$ by multiplication of cosets on the right and this action preserves the coloured edges.
For $\mathcal J
\subseteq \mathcal I$, set $P_{\mathcal J}
= \langle P_k \mid k \in \mathcal J\rangle$ and
$\mathcal C_\mathcal J = (P_{\mathcal J}/T, (P_{\mathcal J}/P_k), k \in \mathcal J)$. Then $\mathcal C_\mathcal
J$ is the $\mathcal J$-connected component of $\mathcal C$ containing the vertex $T$.

We will show  $\mathcal C$  locally resembles the corresponding chamber system in  $\F_4(2)$. This means that for $\sigma \subset \mathcal I$ with $|\sigma|=2$ we will show $P_{\sigma}/O_2(P_{\sigma})$ is isomorphic to the corresponding group in $\F_4(2)$. Since $U_1/O_2(U_1) \cong \Sp_6(2)$ this is true if $\sigma \subseteq \{1,2,3\}$. Hence we may assume that $4 \in \sigma$. There are two possibilities for the relationship between $P_2$ and $P_4$ (they are both contained in $U_2$), but we may  have $P_{24}/O_2(P_{24}) \cong \SL_3(2)$ or $P_{24} = P_2P_4$. We shall show that the latter is in fact the case. We will also prove  $P_{14} = P_1P_4$. This is the purpose of the next lemma.

\begin{lemma} \label{buildingF4} The subgroup  $Z_2(T)$ is normalized by $P_{14}$,  $P_{14} = P_1P_4$  and $P_{24}= P_2P_4$.
\end{lemma}

\begin{proof}  Let $V= Z_2(T)$. Then, by Lemma~\ref{p23} (iv),  $V\cap Z(O_2(U_2)) \not \le Z(O_2(U_1))$.

 As $C_{O_2(U_1)/Z(O_2(U_1))}(T)$ has order $2$ by Lemma~\ref{sp62spin} and $|V \cap Z(O_2(U_2))|=2^3$ by Lemma~\ref{sp62natural}, we deduce   $V= (V \cap Z(O_2(U_1)))(V\cap Z(O_2(U_2)))$ has order $2^4$ as $Z(T) =Z(O_2(U_1))\cap Z(O_2(U_2))$.

Using  Lemmas~\ref{sp62natural} and \ref{sp62spin},   $V\cap Z(O_2(U_1))$ and $VZ(O_2(U_1))$ are both normalized by $P_1$.  Set  $$W= \langle V{^{P_{1}}}\rangle.$$ Then, as the set $V^{P_1}$ has size at most $3$, $W /(V\cap Z(O_2(U_1)))$ has order at most $2^3$ and  $W=V(W\cap Z(O_2(U_1)))$.  Since $(W\cap Z(O_2(U_1)))/(V\cap Z(O_2(U_1)))$ has order at most $2^2$, Lemma~\ref{sp62natural} implies  $O^2(P_1)$ centralizes  $(W\cap Z(O_2(U_1)))/(V\cap Z(O_2(U_1)))$. But then $W/(V\cap Z(O_2(U_1)))$ is centralized by $O^2(P_1)$. Thus $W=V$.  We may apply the same argument to $U_2$ to see that  $P_4$ also normalizes $V$ and so deduce that $P_{14}$ acts on $V$ which has order $2^4$.

We have $[V,O_2(P_1)] \le Z(O_2(U_1)) \cap Z(O_2(U_2))= Z(T)$. Hence, as $[V,O_2(P_1)]$ is normalized by $P_1$, $[V,O_2(P_1)]= \langle x_1\rangle$. Similarly $  \langle x_2\rangle = [V,O_2(P_4)]$. Therefore $O_2(P_1) \cap O_2(P_4)$ centralizes $V$ and has index $4$ in $T$. Thus  $C_T(V)= O_2(P_1) \cap O_2(P_4)$. In particular, $O_2(P_1)$ acts as a transvection on $V$. Hence $C_V(O_2(P_1))$ has order $2^3$ and so $C_V(O_2(P_1))= V \cap Z(U_1)$ and $C_V(O_2(P_4))= V \cap Z(O_2(U_2))$. Because $C_G(V) \le U_1$, we have also shown $C_G(V)= O_2(P_1) \cap O_2(P_4)$.

Set $$D=\langle O_2(P_1)^{N_G(V)}, O_2(P_4)^{N_G(V)}\rangle C_G(V) /C_G(V).$$
Then $D \cap U_1 = P_1$ and, as  $x_1$ has at most 15 conjugates under the action of $D$, $|D| \leq 12\cdot 15$. The structure of $\Alt(8) \cong\GL_4(2)$  therefore shows $D \cong  \SL_2(2)\times \SL_2( 2),$ or  $\mathrm O_4^-(2)\cong \Sym(5)$.

Let $Q_{12}=O_2(P_{12})$, $W_1$ be the preimage of $C_{Z(O_2(U_1))/\langle x_1\rangle}(Q_{12})$ and define $W= W_1V$. Then $W$ is elementary abelian of order $2^5$. Since $V=(V \cap Z(O_2(U_1)))(V\cap Z(O_2(U_2)))$,
\begin{eqnarray*}
[W,Q_{12}]&=&[W_1(V \cap Z(O_2(U_1)))(V\cap Z(O_2(U_2))),Q_{12}]\\& \le &\langle x_1\rangle [(V \cap Z(O_2(U_1)))(V\cap Z(O_2(U_2))),Q_{12}]\\&=&\langle x_1\rangle [(V\cap Z(O_2(U_2))),Q_{12}]\\&\le &\langle x_1\rangle [(V\cap Z(O_2(U_2))),T] \\&=&\langle x_1\rangle [(V\cap Z(O_2(U_2))),O_2(U_1)O_2(P_4)]\\
&=&\langle x_1\rangle [(V\cap Z(O_2(U_2))),O_2(U_1)]=\langle x_1\rangle.\\
\end{eqnarray*}
As $O_2(U_1)/Z(O_2(U_1))$ is a spin module for $\Sp_6(2)$, $$C_{O_2(U_1))/Z(O_2(U_1))}(Q_{12})= WZ(O_2(U_1))/Z(O_2(U_1))$$ by Lemma~\ref{sp62spin}.  We deduce that $W$ is the preimage of $C_{O_2(U_1)/\langle x_1\rangle}(Q_{12})$ and thus $W$ is normalized by $P_{12}$. Since $Z(O_2(U_1)) \cap Z(O_2(U_2)) = Z(T)$, we have $WZ(O_2(U_2))/Z(O_2(U_2))$ has order $2^2$. It follows from Lemma~\ref{sp62spin} that $O^2(P_{4})$ centralizes $WZ(O_2(U_2))/Z(O_2(U_2))$. Let $W_2= \langle W^{P_4}\rangle$. Then $W_2= W(W_2 \cap Z(O_2(U_2)))$. Since $W/V$ has order $2$, we infer that $W_2/V$ has order at most $2^3$. Thus $(W_2 \cap Z(O_2(U_2)))/(V\cap Z(O_2(U_2)))$ has order at most $2^2$. It follows from Lemma~\ref{sp62natural} that  $(W_2 \cap Z(O_2(U_2)))/(V\cap Z(O_2(U_2)))$ is centralized by $O^2(P_4)$. Therefore $W/V$ is normalized by $TO^2(P_4)= P_4$. This shows that $W$ is normalized by $P_{124}$. Notice that along the way  we have shown that $P_{24}=P_2P_4$.

Suppose that $P_{14}/O_2(P_{14})\cong \mathrm O_4^-(2)$. Then $P_{14}$ acts irreducibly on $V$ and so, as $P_{12}$ does not normalize $V$,  $W$ is an irreducible $P_{124}$-module.  As $P_{14}$ has orbits of length $10$ and $5$ on $V$ and $Z(T) \leq V$, we have that $P_{14}$ does not centralize any element in $W \setminus V$ and so $P_{14}$ acts transitively on the 16 elements of $W \setminus V$. This means the orbits of $P_{14}$ on the involutions of $W$  have lengths 5, 10 and 16.  Since $5$ divides the order of $D$, we get that the number of conjugates of $x_1$ under $P_{124}$ is divisible by $5$ and, as $|x_1^{P_{12}}|=10$,  we conclude $|x_1^{P_{124}}|=10$ or 15.
But then $V = \langle x_1^{P_{124}} \rangle$, contradicting the fact that $P_{124}$ acts  irreducibly on $W$. Hence $P_{14}/O_2(P_{14})\cong \SL_2(2)\times \SL_2(2)$ with $P_{14}=P_1P_4$ and this concludes the proof of the lemma.
\end{proof}

\begin{proof}[Proof of Theorem~\ref{P=F4}]  Using Lemma~\ref{buildingF4} and the observations before the lemma  yields that the chamber systems $\mathcal C_{1,2}$, $\mathcal C_{3,4}$ are projective planes, $\mathcal C_{2,3}$ is a generalized quadrangle and in both cases the  parameters are $3,3$ and the remaining  $\mathcal  C_J$ with $|J| = 2$ are all complete bipartite graphs again with parameters $3,3$. Thus $\mathcal C$ is a chamber system of type $\F_4$ (see \cite{local}) in which all panels have $3$ chambers.
Since $U_1/O_2(U_1)\cong \Sp_6(2) \cong U_2/O_2(U_2)$,  we have $\mathcal C_{1,2,3}$  and $\mathcal C_{2,3,4}$ are the  $\Sp_6(2)$-building.
Hence, as each connected rank $3$ residue  of $\mathcal C$ is a building of type $\mathrm C_3$
 and all the rank $2$  residues of $\mathcal C$ are Moufang polygons, applying \cite[Corollary 3]{local}  yields
 that the universal covering   $\pi : {\mathcal{C}}^\prime \longrightarrow {\mathcal{C}}$  has
$\mathcal{C}^\prime$ a building of type $\F_4$ which also has three chambers on each panel.
  By \cite[Proof of Theorem 10.2 on page 214]{Tits} this building is uniquely determined by the two
  residues of rank three with connected diagram. Thus $\mathcal C^\prime$ is isomorphic to the $\F_4(2)$ building and the type preserving automorphism group $F$ of $\mathcal C'$ is isomorphic to $\F_4(2)$.
Since $\mathcal C'$ is a $2$-cover of $\mathcal C$,  there is a subgroup $U$ of
$F$ such that $U$ contains $U_1$ and $U/D \cong P$ for a suitable normal subgroup
 $D$ of $U$. As $U_1$ is isomorphic to a maximal parabolic subgroup of $F$,  we deduce that $U = F$ and $D=1$.  Thus $P \cong F$.
 \end{proof}

\section{The structure of $M$}

 From now on we suppose that $G$ is a group which satisfies the assumptions of Theorem ~\ref{MT}. We  set $M = N_G(Z)$.  So $C_G(Z)$ has index at most $2$ in $M$. Let $S \in \syl_3(M)$
and $Q= F^*(M)= O_3(M)$.

\begin{lemma}\label{basic} We have $Z=Z(S)=Z(Q)$, $N_G(S)\le M$ and $S \in \syl_3(G)$.
\end{lemma}

\begin{proof}  Since $C_M(Q) \le F^*(Q)= Q$, we have that $Z= Z(Q)=Z(S)$.  Therefore  $N_G(S) \le N_G(Z)= M$ and, in particular,
$S \in\Syl_3(N_G(S))\subseteq \Syl_3(G)$.
\end{proof}

Let $R^*$ be a normal subgroup of $C_H(Z)$ such that $R^*/Q\cong \Q_8 \times \Q_8$ and let  $R \in
\Syl_2(R^*)$. We have that $M/Q$ embeds into $\Out(Q)$ and $\Out(Q)$ is isomorphic to  $\GSp_4(3)$
by \cite[III(13.7)]{Hu}. We now locate $M/Q$ in $\Out(Q)$. We will show that $M/QR$ is isomorphic to $\Sym(3)$ or $2 \times \Sym(3)$, more precise  information will be presented in Lemma~\ref{structM}. The next lemma provides our initial restriction on the structure of $M$.

\begin{lemma}\label{U6F4}
We have that $M/Q$ normalizes $R^*/Q$ and is isomorphic to a subgroup of  the subgroup $\mathbf M$  of
$\GSp_4(3)$ which preserves a decomposition of the natural $4$-dimensional symplectic space over $\GF(3)$ into a
perpendicular sum of two non-degenerate $2$-spaces. Furthermore, $R/Q$ maps to $O_2(\mathbf M)$.
\end{lemma}

\begin{proof} See \cite[Lemma 3.1]{PS1}.
\end{proof}

We next  introduce  a substantial amount of notation. We will use this for the remainder of the paper. We note now that the subgroups $Q_1$ and $Q_2$ defined below will be shown to have order $3^3$ in Lemma~\ref{Qaction}.

\begin{notation}\label{nota}
 \begin{enumerate}  \item
Define $R_1$ and $R_2$ to be the two  subgroups of $R$ isomorphic to $\Q_8$ which map to normal subgroups of
$C_{\mathbf M}(Z(R)Q/Q)$.
 \item For $i=1, 2$, let  $r_i \in Z(R_i)^\#$  and $K_i = C_G(r_i)$.
\item  For $i=1,2$, define $$Q_i= [Q,R_i].$$
\item  For $i=1,2$, let $A_i \le Q_i$ be a fixed $S$-invariant subgroup of $Q_i$ of order $3^2$ and set $A=A_1A_2$.
\item  For $i=1, 2$, we let $$\langle \rho_i \rangle \le A_i$$ be such that $\langle \rho_i\rangle$ is inverted by
$r_i$.
\item  Set $J = C_S(A)$ and $L = N_G(J)$.
\end{enumerate}
\end{notation}

Most of this paper is devoted to the determination of $K_1$ and $K_2$. We will show that $K_i$ is similar to a $2$-centralizer in a group of type  $\F_4(2)$ as defined in Definition~\ref{F4cent} and, for $T \in \syl_2(K_1)$, show that $K_1,K_2$ and $T$ is an $\F_4$ set-up. We then use Theorem~\ref{P=F4} to obtain a subgroup $P \cong \F_4(2)$ of $G$.  Our interim goal to achieve this objective is to show that $C_G(\rho_i)$ is isomorphic to  the corresponding centralizer in $\F_4(2)$ or $\Aut(\F_4(2))$. We eventually do this in Lemma~\ref{ItsSp62}.  However we begin more modestly by determining the precise structure of $M$.

\begin{lemma}\label{Qaction} The following hold.
\begin{enumerate}
\item $|S/Q|\le 3^2$.
\item $Q_1=C_Q(r_2)$ and $Q_2= C_Q(r_1)$ and both are normal in $S$; and
\item $Q_1 \cong Q_2 \cong 3^{1+2}_+$, $[Q_1,Q_2]=1$ and $Q=Q_1Q_2$;
\item $A$ is elementary abelian of order $3^3$.
\end{enumerate}
In particular, $Q$ has exponent $3$.
\end{lemma}

\begin{proof}
Part (i)  follows  from Lemma~\ref{U6F4}.

 That $Q_1$ and $Q_2$ are normalized by $S$ follows  from the action of $M$ on $Q$ as $R_1 Q/Q$ and $R_2 Q/Q$ are normalized by $S/Q$.

 For $i=1,2$, we have that $C_Q(r_i)$ and $Q_i=[Q,r_i]$ commute by the Three Subgroup Lemma. Since $Q_i$
has order $3^3$ it follows  that $Q_i\cong 3^{1+2}_+$. As $r_1r_2$ inverts $Q/Z$, $r_2 $ inverts
$C_{Q/Z}(r_1)$ and so $C_Q(r_1)= Q_2$ and $C_Q(r_2)=Q_1$. In particular, $Q_1$ and $Q_2$ commute and $Q=Q_1Q_2$. This proves (ii) and (iii). Finally (iv) follows from (ii) and (iii).
\end{proof}

\begin{lemma}\label{ActionQ} Every element of $Q$ is $M$-conjugate to an element of $A$.
\end{lemma}

\begin{proof} It suffices to prove that every element of $Q/Z$ is conjugate to an element
of $A/Z$. Let $w\in Q/Z$. Then $w=x_1x_2$ where $x_i \in Q_i/Z$ by Lemma~\ref{Qaction} (iii). Since, from the definition of  $A$, for  $i=1, 2$, $(A\cap Q_i)/Z= A_i/Z$ has order $3$ and $R_i$ acts transitively on $Q_i/Z$, there
exists $s_i\in R_i$ such that $w^{s_1s_2} = x_1^{s_1}x_2^{s_2} \in A/Z$. This proves the claim.
\end{proof}

Recall that by hypothesis  $Z$ is not weakly closed  in $Q$. Hence there is a $g\in G$  such that $Y= Z^g \le Q$ and $Y \neq Z$.  We set
\begin{eqnarray*}
V&=&ZY;\\
H&=&\langle Q, Q^g\rangle; \text{ and } \\
W &=& C_{Q^g}(Z)C_Q(Y).\\
\end{eqnarray*}
Notice that $C_Q(Y)$ normalizes $C_{Q^g}(Z)$ and so $W$ is indeed a subgroup of $G$. Because of
Lemma~\ref{ActionQ} we may and do suppose that $V \le A$. In particular, $V$ is normalized by $S$.
Before we continue  our study of $M$, we  investigate $H$.

\begin{lemma}\label{Z weak Q}  The
following statements hold.
\begin{enumerate}
\item $S> Q$;
\item $Q \cap Q^g$ is elementary abelian of order $3^3$ and is a normal subgroup of $S$;  \item $W=C_Q(Y)C_{Q^g}(Y)$ is a normal subgroup of $H$, $H/W\cong \SL_2(3)$, $WQ \in \syl_3(H)$  and $W/(Q\cap Q^g)$ is a natural $H/W$ module;
\item for $i=1,2$, $V \cap Q_i=Z$ and $A \neq Q\cap Q^g$;
\item $A=[Q,W]\le W$, $A/Z= C_{Q/Z}(S)= C_{Q/Z}(W)$ and $A$ is normal in $N_G(S)$; and
\item for $i=1,2$,  $[WQ/Q,R_iQ/Q]\neq 1$.
\end{enumerate}
\end{lemma}

\begin{proof}  As $Q$ is extraspecial, $C_Q(Y)$ is non-abelian of order $3^4$. By Lemma~\ref{basic},
$M^g/Q^g$ has Sylow $3$-subgroups of order at most  $9$ and   $C_Q(Y)\le M^g$ so we have $Z=C_Q(Y)' \le Q^g$. In
particular  we now have $S> Q$ for else $C_Q(Y) \le Q^g$ and then $Z= C_Q(Y)' \le (Q^g)'= Y$ which is a
contradiction. In particular, (i) holds.

Since $\Phi(Q\cap Q^g) \le Z \cap Y= 1$,  $Q \cap Q^g$ is elementary abelian.

Because  $V \le Q\cap Q^g$,  we have $[V,Q]= Z$ and $[V,Q^g]=Y$ and so
$H$ normalizes and acts non-trivially on
$V$  with $H/C_H(V) \cong \SL_2(3)$.

Turning our attention to $W$, we have  $$[W,Q] = [C_Q(Y)C_{Q^g}(Z),Q] = Z[C_{Q^g}(Z),Q].$$ Since
$[[C_{Q^g}(Z),Y],Q]=1=[Q,Y,C_{Q^g}(Z)]$, the Three Subgroup Lemma implies that $[C_{Q^g}(Z),Q]\le C_{Q}(Y) \le
W$. Therefore $$[Q,W]\le C_Q(Y)\le W$$ and, similarly, $[W,Q^g]\le C_{Q^g}(Z)\le W$. Hence $H$ normalizes $W$ and
of course $W \le C_G(V)$.

As $[C_H(V),Q]\le C_Q(V)=C_Q(Y)\le W$, $H/W$ is a central extension of $\SL_2(3)$. Since $H$ acts transitively on
the four subgroups of order $3$ in $V$, and each such subgroup determines uniquely a subgroup of $H$ we have that
$Q^H$ has exactly $4$ members. Now $O^{3}(H)W/W$ is a central extension of a nilpotent group and is thus
nilpotent. Let $T$ be a Sylow $2$-subgroup of $O^3(H)$. Then as $O^3(H)/W$ is nilpotent, $Q$ normalizes and does
not centralize $T$. It follows that $H= WTQ$ and then the action of $Q$ on $T$ and the fact that
$T/C_T(V)\cong\Q_8$ implies that $T \cong \Q_8$ and that $H/W \cong \SL_2(3)$, as by \cite[Satz V.25.3]{Hu} the
Schur multiplier of a quaternion group is trivial.

Using that $O^3(H)$ acts
transitively on $V^\#$, we see that $O^3(H)$ does not normalize any non-trivial subgroup of $(W\cap Q)/(Q\cap
Q^g)$.

Assume $Q \cap Q^g = V$. Then $|W|=3^6$. As $W'\le  V$, $W$ is generated by groups of exponent $3$ and $W$ is
non-abelian, we have $\Phi(W) = V$. Let $f \in H$ be an involution. Then $fW \in Z(H/W)$ and, by Burnside's
Lemma, $f$ does not centralize $W/\Phi(W)$ and neither does it invert $W/\Phi(W)$, for then, as $f$ inverts $V$,
$W$ would be abelian. Therefore, setting $W_0= C_W(f)V$, we have $W_0>V$. Then, as the faithful representations
of $\SL_2(3)$ in characteristic $3$ have even dimension and the minimal faithful representation for $\PSL_2(3)$
is $3$, $|W_0/V|=3^2$ and $W_0$ is centralized by $O^3(H)$ and normalized by $Q$; in particular, $Q\cap W_0 \le
V$ by the comments at the end of the last paragraph. But then $(W\cap Q)W_0= W_0(W\cap Q^g)= W$ which means that
$$[W, Q] =[W_0,Q][W\cap Q,Q]\le W_0.$$ Consequently $O^3(H)$ centralizes $W/V$    which is a contradiction as we
have already remarked that $f$ does not centralize $W/V$. Therefore $Q\cap Q^g
> V$.

 Since $Q\cap Q^g$ is abelian and $Q$ is extraspecial of order $3^5$, we now have
that $|Q\cap Q^g|= 3^3$ and $W/(Q\cap Q^g)$ is a natural $\SL_2(3)$-module. This completes the proof of the first
two statements in (ii) and all of (iii).

Since $H$ acts two transitively on the non-trivial cyclic subgroups of $V$, $N_G(V)= (N_M(V)\cap N_{M^g}(V)) H$ and therefore  $N_G(V)$ normalizes $Q\cap Q^g$.  From the choice of $V \le A$, we have $S \le N_G(V)$. This is the last statement in (ii).

Suppose that $V \le Q_i$ for some $i\in \{1,2\}$. Then $C_M(V) \ge R_{3-i}$ and so $R_{3-i}$ acts on $Q\cap Q^g$.
Since $|Q\cap Q^g:V|= 3$, we obtain $Q\cap Q^g \le C_Q(r_{3-i})= Q_i$  contrary to $Q\cap Q^g$ being elementary
abelian of order $3^3$. Hence $V$ is not contained in $Q_i$ for $i=1,2$. If $A=Q\cap Q^g $, then
$$Y=[A,C_{Q^g}(Z)]\le[A,S]= Z,$$ which is impossible. Hence we also know that $A \neq Q\cap Q^g$.  Thus (iv) holds.

If $[Q_1,W]\le Z$, then $[Q,W]=[Q_1,W][Q_2,W]\le A_2$. Therefore using (iv), $$[C_Q(V),W]=  [C_Q(V),C_{Q^g}(V)]Z
\le Q\cap Q^g \cap A_2= Z.$$ Since $|Q\cap Q^g|= 3^3$ by (ii),  $Y=[Q\cap Q^g, C_{Q^g}(V)]\le [Q,W]=Z$ which is
impossible. Thus $[Q_1,W] = A_1$ and similarly $[Q_2,W] = A_2$. Now  $[Q,W]= A$ and consequently $[Q,S]=
A$. This proves (v).

Finally, suppose that $[WQ,R_1Q] \le Q$. Then $[Q_1,W] \le A_1$ and is $R_1$-invariant. Hence $[Q_1,W]\le Z$ and
this contradicts (v). Thus $[WQ,R_1Q] \not \le Q$ and (vi) holds.
\end{proof}

Now we are in a position to determine $M$. For this set
  $$M_0 = RQ$$ and let $f$ be an involution in $H$. Then $f$ inverts $V$ and
thus $f \in M$. We refine our choice of $R$ so that $R\langle f \rangle$ is a Sylow $2$-subgroup of
$M_0S\langle f \rangle$.

\begin{lemma}\label{structM} The following hold.
\begin{enumerate}
\item $S= WQ$ and $|S/Q|=3$; and
\item One of the following holds:
\begin{enumerate}
\item $M= M_0S\langle f \rangle$, $C_M(Z)= M_0S$ and $M/M_0  \cong \Sym(3)$; or

\item  $|M:  M_0S\langle f\rangle|=2$, $C_M(Z) =  M_0S\langle  t \rangle$ where $t$ is an involution which
exchanges $R_1$ and $R_2$, centralizes $V$ and inverts $SM_0/M_0$ and  $M/M_0 = \langle t,f\rangle  SM_0/M_0 \cong
2 \times \Sym(3)$ with centre $\langle tf\rangle M_0/M_0$.
\end{enumerate}
\end{enumerate}
\end{lemma}

\begin{proof} We have seen in Lemma~\ref{Z weak Q} (i) and (v)  that $|S/Q| \ge 3$ and  $A/Z=C_{Q/Z}(S)= C_{Q/Z}(W)$.

Suppose that  $|S/Q|= 3^2$ and assume that $B$ is an abelian subgroup of $Q$ which is normal in $S$ of order
$3^3$ with $B\neq A$. For $i=1,2$, let $s_i \in S$ be such that $[s_i,R_{3-i}]\le Q$. Then $[B,s_i]\le B
\cap A\cap Q_i \le A_i$. Thus if $s_i$ does not centralizes $B/Z$, then $A_i \le B$. Since $S= \langle
s_1,s_2\rangle$ and $B\neq A$, without loss of generality we may suppose that $A_1 \le B$ and $[B,s_2]\le
Z$. In particular, $B \le Q_1 A$ as $C_{Q/Z}(s_2)= Q_1A/Z$. But then $A_1$ is centralized by $AB= Q_1A$ and we have a contradiction
as $Z(Q_1A)= A_2$. Thus, if $B\le Q$ is a normal abelian subgroup of $S$ of order $3^3$, then $B=A$. Taking
$B = Q\cap Q^g$, we now have that $Q\cap Q^g= A$ a possibility which is eliminated  by Lemma~\ref{Z
weak Q} (iv). Thus $|S/Q|=3$. This proves (i).

We know that $f$ inverts $W/(Q\cap Q^g)$ and so $WQ/Q$ is inverted by $f$. In particular, $M_0S\langle
f\rangle/M_0 \cong \Sym(3)$. If $M=M_0S\langle f\rangle$, then (ii)(a) holds. So assume that $M>M_0S\langle
f\rangle$. As $M$ inverts $Z$, we have $M= C_M(Z)\langle f\rangle$. Since, by Lemma~\ref{U6F4}, $C_M(Z)/Q$ is
isomorphic to a subgroup of $\Sp_2(3)\wr 2$ and since $S/Q$ has order $3$, Lemma~\ref{Z weak Q} (vi) implies that
$C_M(Z)/M_0 \cong 3 \times 2$ or $\Sym(3)$. Especially, there  is a $2$-element   $t\in C_M(Z)\setminus M_0$ which
normalizes $R\langle f\rangle$ and swaps $R_1$ and $R_2$. Because $R\langle t \rangle$ is isomorphic to a Sylow
$2$-subgroup of $\Sp_2(3)\wr 2$, we may as well assume that $t$ is an involution  and that $t$ normalizes $S$.

Since $t$ normalizes $S$ and swaps $R_1$ and $R_2$, $t$ also interchanges $Q_1$ and $Q_2$ and normalizes $A$. It
follows that $t$ normalizes $V$. Without loss of generality we may now additionally assume that $t$ normalizes
$Y$. Thus $t$ normalizes $Q\cap Q^g$ as well as $A$. Since $t$ centralizes $Z$, $[Q,t]$ is extraspecial of order
$3^{1+2}$. Hence either $t$ centralizes $V$ and $Q/C_Q(V)$ or $t$ inverts $V/Z$ and $Q/C_Q(V)$. Multiplying $t$
by $r_1r_2$, we may assume that $t$ centralizes $V$. If $S/Q$ is centralized by $t$, we now have $S/C_Q(V)$ is
centralized by $t$. However, as $[Q,S](Q\cap Q^g)= C_Q(V)/(Q\cap Q^g)$, we see that $S/(Q\cap Q^g)$ is
extraspecial and since $t$ centralizes $S/C_Q(V)$, Burnside's Lemma implies that $t$ centralizes $S/(Q\cap Q^g)$.
Then $t$ also centralizes $Q$ which is a contradiction.  Hence $t$ inverts $S/Q$ and therefore $C_M(Z)/M_0$
has the structure described in (ii)(b).
\end{proof}

\section{The structure of $L = N_G(J)$}

In this section  we continue to use the   notation introduced in \ref{nota}. We also recall $H = \langle Q , Q^g\rangle$ and $f$ is an involution in $H\cap M$ which inverts $Z$.

We will show that $J$ is the Thompson subgroup of $S$ and determine $L = N_G(J)$.

Set $$H_1 = H^{r_1}, W_1= W^{r_1} \mbox{ and }V_1= V^{r_1}.$$

\begin{lemma}\label{HnotH1} We have $W \not= W_1$ and $H \not= H_1$.
\end{lemma}

\begin{proof} Notice that $r_1$ inverts $A_1/Z$ and centralizes $A_2/Z$. Therefore, $V^{r_1} \neq V$.
Since $$W' = [C_Q(V),C_{Q^g}(V)]V  \le Q\cap Q^g \cap [Q,W]=  Q\cap Q^g\cap A= V,$$
we see  $W'= V$ and $W_1'= V_1$. Thus $W$ and $W_1$ are not equal and so also $H \neq H_1$.
\end{proof}

\begin{lemma}\label{3classes} For $i=1,2$, we have $\rho_i$ is not $G$-conjugate to an element of $Z$. In
particular, $A$ contains exactly seven $G$-conjugates of $Z$.
\end{lemma}

\begin{proof} Assume that $i\in \{1,2\}$ and set  $U= \langle \rho_i  \rangle Z$. If $\rho_i$ was $G$-conjugate to an element of $Z$, then
Lemma~\ref{Z weak Q}(iv) applied with $U$ in place of $V$   yields $\langle \rho_i\rangle = U\cap Q_i \le Z$
which is absurd. Thus $\rho_i$ is not $G$-conjugate to an element  $Z$. Since $V\cup V_1 \subset A$, we now see  $A$ contains exactly seven $G$-conjugates of $Z$ three  $Q$-conjugates of $\langle\rho_1\rangle$ and three $Q$-conjugates of $\langle \rho_2\rangle$.
\end{proof}

 We can now describe  the structure of $L$.

\begin{lemma}\label{NJ} The following hold.
\begin{enumerate}
\item $J= J(S)$ is elementary abelian of order $3^4$.
\item $L$ controls $G$-fusion of elements of $J$.
\item $J= C_G(J)$.
\item $L$ preserves  a quadratic form $\mathrm q$ of $+$-type on $J$ up to similarity.
\item Set $L_* = \langle H,H_1, r_1,r_2\rangle$. Then $L_*/J \cong \GO_4^+(3)$ and either
\begin{enumerate}
\item if $M= M_0S\langle f\rangle$, then $L=L_*$; or
\item if $M> M_0S\langle f\rangle$, then $L/J\cong \mathrm {CO}_4^+(3)$. (Here   $\mathrm {CO}_4^+(3)$
is the group which preserves  $\mathrm q$ up to similarity.)
\end{enumerate}
\end{enumerate}
\end{lemma}

\begin{proof}  By construction $A$ is elementary abelian and
so  $A \le C_Q(V)  \le W$ and $A \le C_Q(V_1) \le W_1$. Since $S$ centralizes $A/Z$ and since in $\GL_3(3)$  such
a centralizer has order $18$, we infer that $J=C_S(A)$ has order $3^4$. Since $A$ has index $3$ in $J$,  $J$ is abelian.  Suppose that $B$ is an abelian subgroup of $S$ of order at least $3^4$. We may assume that
$B \ge Z$. Thus by Lemma~\ref{structM}, $B \cap Q$ is an abelian subgroup of $Q$ of order at least $3^3$ and
hence of order exactly $3^3$.  Using that $(B\cap Q)/Z$ is centralized by $QB=S$, Lemma~\ref{Z weak Q} (iii)
yields $B\cap Q= A$. But then $B \le C_S(A)=J$ and we have $B=J$. Hence $J= J(S)$ is the Thompson subgroup of
$S$. Since $J$ centralizes $V$, $J \le S\cap C_G(V)= W$. Thus $J= J(W) $ and similarly $J= J(W_1)$. In
particular, $L \ge \langle H, H_1\rangle N_G(S)$. Since $J$ contains $A$, if $J$ is not elementary abelian, then
$\Phi(J)= Z$. But then $Z$ is normalized by $H$, which is a contradiction as $H$ acts irreducibly on $V$. Thus
$J$ is elementary abelian. This proves (i). Part (ii) follows from \cite[37.6]{AschbacherFG} as $J$ is abelian.

We have that $C_G(J) \le C_G(Z)<M$. Since $J$ acts non-trivially on both $R_1Q/Q$ and $R_2Q/Q$, and $JM_0/M_0$ is
inverted by $t$ when $M> M_0S\langle f\rangle$ (see Lemma~\ref{structM} (ii)), we have $C_M(J)\le S\langle
r_1,r_2\rangle$. Since $r_1Q$ and $r_2Q$ act non-trivially on $A/Z$, we have $C_G(J) \le S$. Hence $J\le C_G(J)=
C_S(J)\le C_S(A)\le J$ and this proves (iii).

Define $$\mathcal S (J) = \{j \in J^\#\mid j^{l}\in Z \text{ for some } l \in L\}.$$

Consider $S/J = Q_1Q_2J/J$.  Then $S /J \in \syl_3(L_*/J) \subseteq \syl_3(L/J)$.  We have $[J,Q_1]= A_1=
C_J(Q_2)$ and $[J,Q_2]= A_2= C_J(Q_1)$. In addition, $[J,S]= [J,Q]= [W,Q]=A$  and $C_J(S)= Z$.

Now $\langle Z^{L_*}\rangle \ge \langle Z^H\rangle \langle Z^{H_1}\rangle = VV_1= A$ and, as $A \not \le Q\cap
Q^g$, $A$ is not normalized by $H$. Hence $\langle Z^{L_*}\rangle=J$  and, in particular, $L_*$ and,
consequently, $L$ acts irreducibly on $J$. Thus there are members of $\mathcal S(J)$ in $J\setminus A$. By
Lemma~\ref{3classes} there are exactly $14$ elements of $\mathcal S(J)$ in $A$ and  in $J\setminus A$ there are a
multiple of 18 such elements. Thence $|\mathcal S(J)| = 14+ n\cdot 18$ for some integer $n\ge 1$. Since $|J|= 3^4$,
using the fact that $|\mathcal S(J)| $ divides $|\GL_4(3)|$ we infer that $|\mathcal S(J)|= 32$.

Using Lemma~\ref{GO4}  with $\langle a \rangle = Q_1J/J$ and $\langle b \rangle = Q_2J/J$, yields that $S$ preserves a
quadratic form with any element of $\mathcal S(J)$ as a singular vector. Since $S/J$ contains $W_1/J$ and $W_2/J$
which both act quadratically on $J$ with $[J,W]= [J,J(Q\cap Q^g)]= [J,(Q\cap Q^g)]=V$ and $[J,W]= [J,W]^{r_1}=
V_1$ we see that for any such  form $V$ and $V_1$ would consist of singular vectors. It follows that $\mathcal S(J)$
is the set of singular vector of a $+$-type quadratic form on $J$. Since this set is by design invariant under
the action of $L$, we have $L/J $ is isomorphic to a subgroup of $\mathrm {CO}_4^+(3)$ by Lemma~\ref{quadratic
form}. Thus (iv) is true. Now $HH_1$ contains $S=WW_1$ which is a Sylow $3$-subgroup of $G$, $H$ acts
irreducibly on $V$ and $H_1$ acts irreducibly on $V_1$, it follows that $HH_1/J \cong \Omega_4^+(3)$. Conjugation
by $r_1$ exchanges $H$ and $H_1$,  $\langle r_1r_2\rangle H_1/W_1 \cong  \GL_2(3)$ and so we infer that $L_*/J
\cong \mathrm {GO}_4^+(3)$ and $L_*$ is normal in $L$. By the Frattini  Argument, $L= N_L(S) L_* = N_M(S)L_*$ and
so (v) holds.
\end{proof}

\begin{lemma}\label{fusion3elts}
We have   $\rho_1$ is $G$-conjugate to $\rho_2$ if and only if $SR\langle f \rangle$ has index $2$ in $M$.
\end{lemma}

\begin{proof}
This  is a consequence of Lemma~\ref{NJ}(ii) and (v).
\end{proof}

Recall the notation introduced in ~\ref{o4} and ~\ref{type}.

\begin{lemma}  The sets $\mathcal P(J)$ and $\mathcal
M(J)$ are fused in $L$  if $L > L_*$  and we have $|\mathcal S(J)|= 16$, $|\mathcal P(J)|=|\mathcal M(J)|= 12$.
\end{lemma}

\begin{proof} This follows directly from Lemma~\ref{NJ}.
\end{proof}

%

\begin{lemma}\label{I1}  For $i=1,2$, $C_L(r_i) = C_{L_*}(r_i)$, $[J,r_i]= \langle \rho_i\rangle$, $|C_J(r_i)|=3^3$ and $C_L(r_i)/C_J(r_i) \langle r_i\rangle \cong \GO_3(3) \cong 2 \times \Sym(4)$.
\end{lemma}

\begin{proof}  We have that $|C_S(r_i)| = 3^4$ and $r_i$ inverts $Q_iJ/J$. Hence $|C_J(r_i)|= 3^3$.
It follows that both $r_1$ and $r_2$ are reflections on $J$. If $L> L_*$, then $r_1^t= r_2$ and so
$C_L(r_i) = C_{L_*}(r_i)$.  Since $r_1$ and $r_2$ are reflections and since $L_*/J \cong \GO_4^+(3)$ by
Lemma~\ref{NJ}, we  have $C_L(r_i)/C_J(r_i) \langle r_i\rangle \cong \GO_3(3) \cong 2 \times \Sym(4)$.
\end{proof}

From Lemma~\ref{I1} we have  $[J,r_1] = \langle \rho_1\rangle$ and $[J,r_2] =\langle \rho_2\rangle$ are non-singular 1-dimensional spaces in $J$. We fix notation so that $\langle \rho_1\rangle \in \mathcal P(J)$ and
$\langle \rho_2\rangle \in \mathcal M(J)$.

\begin{lemma}\label{type1} The following hold:
\begin{enumerate} \item $V$ and $V_1$ are of Type S;
\item $A_1$ is of Type DP;
\item $A_2$ is of Type DM;
\item  $\langle \rho_1,\rho_2 \rangle$ is of type $N+$;
\item  $|\mathcal S(C_J(r_1))|=4$, $|\mathcal M(C_J(r_1))|=6$ and $|\mathcal P(C_J(r_1))|=3$; and
\item $|\mathcal S(C_J(r_2))|=4$, $|\mathcal M(C_J(r_2))|=3$ and $|\mathcal P(C_J(r_2))|=6$.
\end{enumerate}
\end{lemma}

\begin{proof} Parts (i)--(iv) are obvious.   By Lemma~\ref{I1} we have that $|C_J(r_i)| = 3^3$ for $i = 1,2$. Since $J$ is a quadratic space of plus type, it follows that
$C_J(r_1)$ has an orthonormal basis consisting of members of $\mathcal P(J)$ and $C_J(r_2)$ has an orthonormal basis
consisting of elements of $\mathcal M(J)$. Thus (v) and (vi) hold.
\end{proof}

\begin{lemma}\label{I2} If $\wt {\rho_{i}} \in C_J(r_i)$ is $L_*$-conjugate to $\rho_i$, then $\langle \rho_i, \wt \rho_i\rangle$
has Type N-. In particular,   $|\mathcal P (\langle \rho_i, \wt \rho_i\rangle)|=|\mathcal M(\langle \rho_i, \wt
\rho_i\rangle)|=2 $.
\end{lemma}

\begin{proof} Suppose that $\wt {\rho_{i}} \in C_J(r_i)$ is $L_*$-conjugate to $\langle \rho_i\rangle$. Then, as $\langle \rho_i\rangle= [J,r_i]$, $\rho_i$ is perpendicular to $C_J(r_i)$. It follows that $\wt {\rho_{i}} $ is perpendicular to $\rho_i$ and this means that $ \langle \rho_i, \wt \rho_i\rangle$ is of Type N-.
\end{proof}

\section{Two $3$-centralizers}

In this section we determine the structure of $C_G(\rho_1)$ and $C_G(\rho_2)$. We first show that these centralisers do not have non trivial normal $3^\prime$-subgroups. Recall the notation of \ref{nota} and that $f\in M$ is an involution inverting $Z$.

\begin{lemma}\label{JSig} $J$ does not normalize any non-trivial $3'$-subgroups.
\end{lemma}

\begin{proof}  Suppose that $Y$ is a non-trivial  $3'$-subgroup normalized by $J$. Then, as every subgroup  of $J$ of order $27$ contains a conjugate
of $Z$ by Lemma~\ref{hyper}, we may assume that $X= C_{Y}(Z)\neq 1$. As $X$ is
normalized by $A = J \cap Q$ and $X$ normalizes $Q$, $[A,X] \le Q \cap X=1$ and hence $X \le
C_M(A) = J $ as $A$ is a maximal abelian subgroup of $Q$. But then $X=1$ which is a contradiction. This proves the lemma.
\end{proof}

\begin{lemma}\label{princeprep} For $i=1,2$,  $C_M(\rho_i) = Q_{3-i}R_{3-i}J\langle fr_i \rangle$
and $C_{C_M(Z)}(\rho_i)/\langle \rho_i \rangle$ is isomorphic to the centralizer of a non-trivial  $3$-central
element in $\PSp_4(3)$ and $Z$ is inverted in $C_M(\rho_i)$.
\end{lemma}

\begin{proof} Since $\rho_i \in A_i\le J$ and since $[Q_1,Q_2]=1$ and
$[Q_i,R_{3-i}]=1$, we certainly have $C_M(\rho_i) \ge Q_{3-i}R_{3-i}J$. Furthermore, $f$ inverts $J$ and so $f$
inverts $\rho_i$ and as $r_i$ also inverts $\rho_i$, we have $C_M(\rho_i) \ge Q_{3-i}R_{3-i}J\langle fr_i \rangle$
which has index either $24$ or $48$ in $M$ dependent upon whether or not $M = RS\langle f \rangle$
respectively. Since $Q_i$ contains twelve $Q$-conjugates of $\langle \rho_i\rangle$, Lemma~\ref{fusion3elts} implies
$C_M(\rho_i) = Q_{3-i}R_{3-i}J\langle fr_i \rangle$.

Because  $r_if$ inverts $Z$, we have $C_{C_M(Z)}(\rho_i)/\langle \rho_i\rangle= Q_{3-i}R_{3-i}J/\langle
\rho_i\rangle$ with $R_{3-i}$ acting faithfully on $Q_{3-i}$. Thus the final statement also is valid by
Lemma~\ref{cen3psp43}.
\end{proof}

In the next two lemmas we pin down two possible structures of $C_G(\rho_1)$ and $C_G(\rho_2)$. In fact in $\F_4(2)$ we have that both are isomorphic to $3 \times \Sp_6(2)$. That this is the case in our group will be proved later in Lemma~\ref{ItsSp62}.

\begin{lemma}\label{eitheror} For $i=1,2 $ either $C_G(\rho_i) \cong 3 \times \Aut(\SU_4(2))$ or $C_G(\rho_i) \cong 3 \times \Sp_6(2)$. Furthermore, $r_i$ inverts $\rho_i$ and centralizes $C_G(\rho_i)/\langle \rho_i\rangle$.
\end{lemma}

\begin{proof}  We consider $C_G(\rho_i)/\langle \rho_i\rangle$.  By Lemma~\ref{princeprep},
$C_{C_M(Z)}(\rho_i)/\langle \rho_i\rangle$ is isomorphic to
 a $3$-centralizer in $\PSp_4(3)$. Since $J/\langle \rho_i\rangle$ normalizes
 no non-trivial $3'$-subgroup of $C_G(\rho_i)$ by Lemma~\ref{JSig} and $Z$ is inverted by $fr_i$, we may apply
Theorem~\ref{PrinceThm} to obtain $C_G(\rho_i)/\langle \rho_i\rangle \cong \Aut(\SU_4(2))$ or  $\Sp_6(2)$ or that
$C_G(\rho_i) = C_M(\rho_i)$. The latter possibility is dismissed as $C_{L}(\rho_i)$ has index $2$ in $\langle
\rho_i\rangle C_{L_*}(r_i) $  and so, by Lemma~\ref{I1},   $$C_{L}(\rho_i)\cong 3 \times 3^3:(2 \times \Sym(4))$$ does not
normalize $Z$.

 The Sylow $3$-subgroup of $C_G(\rho_i) $ is $\langle \rho_i\rangle \times Q_{3-i}C_J(r_i)$
and hence the extension $C_G(\rho_i)/\langle \rho_i\rangle$ splits by Gasch\"utz Theorem. Finally we have that
$r_i$ centralizes  $Q_{3-i}J/\langle \rho_i\rangle$ and, as no automorphism of either $\Aut(\SU_4(2))$ or
$\Sp_6(2)$ of  order $2$ centralizes such a subgroup, we infer that $r_i$ centralizes $C_G(\rho_i)/\langle
\rho_i\rangle$ and of course we also know that $\rho_i$ is inverted by $r_i$.
\end{proof}

\begin{lemma}\label{thesame} We have $C_G(\rho_1) \cong C_G(\rho_2)$.
\end{lemma}

\begin{proof}  By Lemma \ref{eitheror},  $C_G(\rho_1)/\langle \rho_1 \rangle \cong \Sp_6(2)$ or $\Aut(\SU_4(2))$.

Assume that $C_G(\rho_1)/\langle \rho_1 \rangle \cong \Sp_6(2)$. Using Lemma~\ref{type1} (v), we have some $\wt {\rho_1} \in \mathcal P(C_J(\rho_1))$ and as $|\mathcal P(C_J(\rho_1))| = 3$,   $C_{E(C_G(\rho_1))}(\wt {\rho_1})  \cong 3 \times \Sp_4(2)$ from the structure of $\Sp_6(2)$. Therefore $E(C_G(\langle \rho_1, \wt{\rho_1} \rangle)) \cong \Sp_4(2)'$.
Lemma~\ref{I2}, yields that $\Sp_4(2)'$  is involved in the centralizer of a $3$-element in $C_G(\rho_2)$. As there are no such $3$-elements in $\SU_4(2)$ \cite{Atlas},  Lemma \ref{eitheror} implies  $E(C_G(\rho_2))/\langle \rho_2 \rangle \cong \Sp_6(2)$. Hence Lemma~\ref{thesame} holds.
\end{proof}

\section{Building a signalizer in the centralizers of $r_1$ and $r_2$}

In this section we begin the construction $K_i= C_G(r_i)$ for $i=1,2$. We  give a brief overview of our plans for $i = 1$ to guide the reader through the technicalities involved.
Our final aim is to show that $K_1$ is similar to a $2$-centralizer in a group of type  $\F_4(2)$ (see Definition~\ref{F4cent}).
Hence we aim to show that $K_1$ is an
extension of a 2-group by $\Sp_6(2)$. Furthermore this $2$-group is a product of an extraspecial group of
order $2^9$ by an elementary abelian group. Our first aim is to  construct the extraspecial group $\Sigma_1$,
and show that it is normalized by $C_L(r_1)$. Note that $C_J(r_1)\le C_L(r_1)$ and the former group is elementary abelian of  order $3^3$.

We briefly consider the situation in our target group.  In  $\F_4(2)$ there are exactly four maximal subgroups of $C_J(r_1)$ with  centralizers in $\Sigma_1$ which properly contain $\langle r_1\rangle$ and  these maximal subgroups centralize a quaternion group of order eight in $\Sigma_1$.
In our group $G$,  the first problem is to find these quaternion groups. For this we pick a set
of four maximal subgroups of $C_J(r_1)$, which are conjugate to $A_2$. They all contain a conjugate of $\rho_2$. By
Lemma~\ref{eitheror} there are exactly two possibilities for the structure of $C_G(\rho_2)$. Examining these structures shows  $C_{C_G(\rho_2)}(A_2)/\langle \rho_2 \rangle \cong
3^{1+2}_+{:}\SL_2(3)$. Hence $C_{C_G(\rho_2) \cap C_G(r_1)}(A_2)/\langle \rho_2 \rangle \cong \SL_2(3)$. This shows
that $O_2(C_{C_G(\rho_2) \cap C_G(r_1)}(A_2)) \cong \Q_8$, and this is one of the quaternion groups we are looking for. As
$A_2$ has four conjugates under $C_L(r_1)$, we now get a set of four quaternion groups. The problem is  to show
these four quaternion groups generate a 2-group $ \Sigma_1$  which is extraspecial of order $2^9$.
This will be done in Lemma~\ref{sigmai}. Furthermore, the very  construction guarantees us that $C_L(r_1)$ acts on $\Sigma_1$.

We continue to use the notation from \ref{o4}, \ref{type} and \ref{nota}. Additionally we introduce
\begin{notation}\label{notar} For $i=1,2$, $I_i = C_J(r_i)$ and $F_i = C_L(r_i)$.
\end{notation}

Notice that by Lemma~\ref{I1},  $F_i$ acts on $I_i$ and $F_i/I_i \langle
r_i\rangle\cong 2 \times \Sym(4)$. As explained above we intend  to determine a large signalizer for $I_i$, that is  a  $3^\prime$-group which is normalized by $I_i$. We begin with two  easy observations.

\begin{lemma}\label{CSr} For $i=1,2$, $C_{C_M(Z)}(r_i) = Q_{3-i}R_1R_2I_i$ and  $C_S(r_i) = Q_{3-i}I_i\in \Syl_3(C_M(r_i)) \subseteq \Syl_3(K_i)$.
\end{lemma}

\begin{proof}  Obviously $C_{C_M(Z)}(r_i) \ge Q_{3-i}R_1R_2C_J(r_i)$ and then Lemma~\ref{structM} (ii) implies the equality. Therefore, $C_S(r_i) = Q_{3-i}I_i \in \syl_3(C_M(r_i))$ and $Z(C_S(r_i))= Z$. Thus $N_{K_i}(C_S(r_i)) \le N_G(Z)=M$. In particular, $C_S(r_i) \in \syl_3(K_i)$.
\end{proof}

\begin{lemma}\label{fusionr1r2} We have $r_1$ is $G$-conjugate to $r_2$ if and only if $r_1$ is $M$-conjugate to $r_2$.
\end{lemma}

\begin{proof}  Obviously if $r_1$ and $r_2$ are conjugate in $M$ then they are conjugate in $G$.  Suppose then that $r_1=r_2^g$  for some $g\in G$. By Lemma~\ref{CSr}, for $i=1,2$,  $C_S(r_i) \in \Syl_3(C_G(r_i))$ and $Z = Z(C_S(r_i))$. Since $r_1= r_2^g$,  $C_S(r_2)^g \in \syl_3(C_G(r_1))$.  Thus there is $h \in C_G(r_1)$ such that $C_S(r_2)^{gh} = C_S(r_1)$. But then $$Z^{gh}= Z(C_S(r_2))^{gh} = Z(C_S(r_1))=Z$$ which means that $gh \in M$.  Hence $r_1$ and $r_2$ are $M$-conjugate.
\end{proof}

Recall, for $i=1,2$,  $$I_i=C_J(r_i)=J \cap E(C_G(\rho_i))$$ as, by Lemma \ref{eitheror}, $ E(C_G(\rho_i)) = C_{C_G(\rho_i)}(r_i)$.

\begin{lemma}\label{crossover} Suppose that  $\wt \rho_1 \in \mathcal P(I_1)$ and $\wt \rho_2 \in \mathcal M(I_2)$. Then, for $i=1,2$,  in $E(C_G(\wt \rho_i))\langle r_i\rangle$, $r_i$ is an involution which has $\Sp_4(2)'$ as a composition factor of its centralizer. Moreover,  $I_i \cap E(C_G(\wt \rho_i))$ is of Type N-.
\end{lemma}

\begin{proof}  For $i=1,2$,
the definition of $I_i$,  yields $r_i \in C_G(\wt \rho_i)$. Now $r_i$  normalizes $E(C_G(\wt \rho_i))$ and centralizes
$I_i \cap E(C_G(\wt \rho_i))$ which has order $9$.

On the other hand, in $C_G(\rho_i)$, as there are only three conjugates of  $\langle \wt{\rho_i}\rangle$  in
$I_i$ by Lemma~\ref{type1}(v) and (vi), we have that $$C_{E(C_G(\rho_i))}(\wt \rho_i)\approx 3 \times 3^2.\Dih(8)$$ if
$E(C_G(\rho_i)) \cong \SU_4(2)$ and  $$C_{E(C_G(\rho_i))}(\wt \rho_i)\approx 3 \times \Sp_4(2) $$ if $E(C_G(\rho_i))
\cong \Sp_6(2)$. As $I_i \le E(C_G(\rho_i))$, it follows that $$I_i \cap [I_i,C_{E(C_G(\rho_i))}(\wt \rho_i)]$$ is of Type N-.
Now deploying
Lemmas~\ref{sp62facts} and \ref{sp62line} (ii), $C_{E(C_G(\wt \rho_i))}(r_i) \cong \Sp_4(2)$ if
$E(C_G(\wt \rho_i)) \cong \SU_4(2)$ and has shape $2^5.\Sp_4(2)$  when $E(C_G(\wt \rho_i))\cong \Sp_6(2)$.  In particular, the main claim in
the lemma is true. We have already observed that $I_i \cap [I_i,C_{E(C_G(\rho_i))}(\wt \rho_i)]$ has Type N-
and as this group is  $I_i \cap E(C_G(\wt \rho_i))$ we have the last part of the lemma.
\end{proof}

We can now locate the four maximal subgroups of $I_i$, whose centralizers  contain the quaternion groups we are looking for.
 Recall that, for $i=1,2$,  $A_{3-i}= A \cap Q_{3-i}$ is a hyperplane of $I_i$ which with respect to the quadratic form on $J$ is a degenerate
$2$-dimensional subspace which contains one conjugate of $Z$ and three conjugates of $\langle \rho_{i}\rangle$. Therefore
$A_{1}$ has Type DP  and has $A_2$ Type DM in the sense of Notation~\ref{type}. Consequently  the set $A_{3-i}^{F_i}$ has order $4$. We let the four
$F_i$-conjugates of $A_{3-i}$  be $I_i^1=A_{3-i}$, $I_i^2$, $I_i^3$ and $I_i^4$. Then, for $1\le j < k \le 4$, we have
$I_i^j \cap I_i^k $ is an $M$-conjugate of $\langle \rho_{3-i}\rangle$. We further select notation so that $$I_i ^1 \cap
I_i^2 = \langle \rho_{3-i}\rangle.$$ The next lemma  follows immediately from the  2-transitive action of $F_i$ on
the set $\{I_i^1,I_i^2,I_i^3,I_i^4\}$.

\begin{lemma}\label{IijIik} For $1 \le l \le 4$ and $1 \le j < k \le 4$ we have \begin{enumerate}
\item $I_1^l$ has Type DM  and  $I_1^j\cap I_1^k\in \mathcal M(I_1)$; and  \item $I_2^l $ has Type DP and $I_2^j\cap I_2^k\in \mathcal
P(I_2)$.\end{enumerate}\qed
\end{lemma}

With these comments we  have the following lemma directly from Lemmas~\ref{eitheror} and \ref{thesame}.

\begin{lemma}\label{Cij} For $i=1,2$ and for $1\le j< k\le 4$, we have $$C_G(I_i^k\cap I_i^j) \cong  3 \times \Aut(\SU_4(2)) \text{  or }  3 \times \Sp_6(2).$$ Furthermore, the isomorphism type  of $C_G(I_i^k\cap I_i^j)$ does not depend on $i$, $j$ or $k$.\end{lemma}

Recall the Type N+ subgroups of order $9$ are just the non-degenerate subgroups of $J$  of plus type.
\begin{lemma}\label{I1capI2}  $I_1 \cap I_2$  is of Type N+.
\end{lemma}

\begin{proof} We know that $I_1 \cap I_2 = C_J(\langle r_1,r_2\rangle)$ and is consequently non-degenerate.
Since $Z \le I_1 \cap I_2$, it has Type N+.
\end{proof}

The next lemma is an adaption of Lemma~\ref{NJ}(ii) to $K_i$.

\begin{lemma}\label{fusion1} $F_i = N_{K_i}(I_i)$ controls $K_i$-fusion of elements in $I_i$.
\end{lemma}

\begin{proof} By Lemma~\ref{CSr},  $C_S(r_i)\in \syl_3(K_i)$ and thus $I_i $ is the Thompson subgroup of
$C_S(r_i)$ and is elementary abelian. It follows from \cite[37.6]{AschbacherFG} that $N_{K_i}(I_i)$ controls fusion in $I_i$. As $C_G(I_i) \le
M$, we calculate that $C_G(I_i)= J\langle r_i\rangle$. Hence $C_{K_i}(I_i)= I_i\langle r_i\rangle$ and
$N_{K_i}(I_i) =  L \cap K_i = F_i$.
\end{proof}

For $i \in \{1,2\}$ and $1\le j < k\le 4$,
$$E_i^{j,k} = E(C_G(I_i^j\cap I_i^k)).$$ So $E_i^{j,k} \cong \SU_4(2)$ or $\Sp_6(2)$ and we note
again that the isomorphism type of this group does not depend on  $i,j$ or $k$.  At least one potential avenue for
confusion is caused by this notation so please note that $E_i^{j,k}$ does not centralize $r_i$. Rather it
centralizes a conjugate of $r_{3-i}$.  Indeed  $E_1^{1,2}= E(C_G(\rho_2))$  centralizes $r_2$ and $E_2^{1,2}= E(C_G(\rho_1))$ centralizes $r_1$ by
Lemma~\ref{eitheror}.

Notice that $I_i$ is centralized by $r_i$ and so $r_i$ is contained in  $C_G(I_i^j\cap I_i^k)$ and it centralizes $I_i \cap E_i^{j,k}$ and this contains $Z$. It follows that $I_i \cap E_i^{j,k}$ is of Type N+ as it must also be non-degenerate. This means that $r_i$ acts as  an involution of  type $a_2$ on  $E_i^{j,k}$ in the sense of Table 1.  Therefore, Lemma~\ref{sp62facts}(ii) gives the following result:
\begin{lemma}\label{cri} We have
\begin{eqnarray*}C_{K_i}(I_{i}^j \cap I_i^k) &=& C_{C_G(I_i^j \cap I_i^k)}(r_i)\\& \approx& \begin{cases}3 \times 2^{1+4}_+.(\Sym(3) \times \Sym(3))&E_i^{j,k} \cong \SU_4(2) \\3\times
2^{1+2+4}.(\Sym(3) \times \Sym(3))&E_i^{j,k} \cong \Sp_6(2)\end{cases}.\end{eqnarray*}\qed
\end{lemma}

The next lemma now is the key. It shows that the groups $O_2(C_{K_i}(I_j^{i}))$ are quaternion groups of order eight which pairwise commute and so generate an extraspecial group of order $2^9$.

\begin{lemma}\label{Sigmapieces} Assume that $i=1,2$ and  $1\le j< k\le 4$.
\begin{enumerate}
\item For $m \in \{j,k\}$, $I_i^m \cap E_i^{j,k}$ is a $3$-central element of $G$ and of $E_{i}^{j,k}$;
\item $C_{G}(I_i^k) =  (I_i^k\cap I_i^j) \times C_{E_i^{j,k}}(I_k \cap E_i^{j,k}) \approx 3 \times 3^{1+2}_+.\SL_2(3)$;
\item \begin{enumerate}\item  $O_2(C_{K_i}(I_i^j)) \cong O_2(C_{K_i}(I_i^k))  \cong \Q_8$; \item  $O_2(C_{K_i}(I_i^j))O_2(C_{K_i}(I_i^k)) \le O_2(C_{K_i}(I_i^j\cap I_i^k))$ with equality if $E_i^{j,k} \cong \SU_4(2)$; and \item  $[O_2(C_{K_i}(I_i^j)), O_2(C_{K_i}(I_i^k))]=1$; and\end{enumerate}
\item $C_{I_i}(O_2(C_{K_i}(I_i^j))O_2(C_{K_i}(I_i^k)))= I_i ^j\cap I_i^k$.
\end{enumerate}
\end{lemma}

\begin{proof} It suffices to prove  part (i)  for $I_i^1 $ as then the result will follow by conjugating by $F_i$

So consider $I_i^1\cap I_i^2 = \langle \rho_{3-i}\rangle$. Then, by Lemma~\ref{princeprep},  $C_S(\rho_{3-i})= Q_iJ$ and $C_S(\rho_{3-i})' \cap Z(C_S(\rho_{3-i}))= Z$.
Thus $Z \le I_i^1 \cap E_{i}^{1,j}$ is $3$-central in $G$
and in $E_{i}^{1,j}$. Part (i) follows as $F_i$ acts  2-transitively on $\{I_i^j\mid 1 \le j \le 4\}$.

Part (ii) follows from (i) as the centralizer of a $3$-central element in $\Sp_6(2)$ and $\SU_4(2)$ has shape $3^{1+2}_+.\SL_2(3)$.

To deduce  part (iii),  we first note that  $$O_2(C_{K_i}(I_i^k)) \cong O_2(C_{K_i}(I_i^{j})) \cong  \Q_8$$
follows from (ii) as $r_i$ is an involution in $C_G(I_i^{k})$. We have $l \in \{j,k\}$,
$O_2(C_{K_i}(I_i^l))  \le C_{K_i}(I_{i}^j \cap I_i^k) $ and is normalized by $I_i^jI_i^k= I_i$. Since
\begin{eqnarray*}C_{K_i}(I_{i}^j \cap I_i^k) &=& C_{C_G(I_i^j \cap I_i^k)}(r_i)\\& \approx& \begin{cases}3 \times 2^{1+4}_+.(\Sym(3) \times \Sym(3))&E_i^{j,k} \cong \SU_4(2) \\3\times
2^{1+2+4}.(\Sym(3) \times \Sym(3))&E_i^{j,k} \cong \Sp_6(2)\end{cases}\end{eqnarray*}by Lemma~\ref{cri}, it follows that $O_2(C_{K_i}(I_i^l)) \le
O_2(C_{K_i}(I_{i}^j \cap I_i^k))$. Now we apply Lemma~\ref{sp62line}(iii) to see that
$[O_2(C_{K_i}(I_i^k)),O_2(C_{K_i}(I_i^k))]=1$. (Recall that $O_2(C_{\SU_4(2)}(r_i)) \le
O_2(C_{\Sp_6(2)}(r_i))$.)

 Part (iv) follows as  $I_i \cap E_{i}^{j,k}$ acts faithfully on $O_2(C_{K_i}(I_i^j))O_2(C_{K_i}(I_i^k))$.

\end{proof}
We now introduce some further notation

\begin{notation}\label{sigma}  For $i = 1,2$, $1\le k \le 4$, $$\Sigma_i^k= O_2(C_{K_i}(I_i^k))\cong \Q_8$$ and  $$\Sigma_i= \langle \Sigma_i^k\mid 1\le k \le 4\rangle=\langle O_2(C_{K_i}(I_i^k)) \mid 1 \le k \le 4\rangle.$$
\end{notation}
Note that $\Sigma_1^1= O_2(C_{K_1}(A_2)) =  R_1$ and $\Sigma_2^1= O_2(C_{K_2}(A_1)) = R_2$.

\begin{lemma}\label{sigmai} We have $\Sigma_i$ is extraspecial of order $2^9$ and plus
type, $Z(\Sigma_i)= \langle r_i\rangle$ and $F_i/\langle r_i\rangle$ acts faithfully on $\Sigma_i$.
\end{lemma}
\begin{proof} The structure of $\Sigma_i$ follows  from Lemma~\ref{Sigmapieces} (iii) as the generating subgroups commute pairwise. To see the last part is suffices to show that $I_i$ acts faithfully on $\Sigma_i$ as every normal subgroup of $F_i$ which strictly contains $\langle r_i\rangle$ contains $I_i$.
Using Lemma~\ref{Sigmapieces} (iv) we see that $C_{I_i}(\Sigma_i) = \bigcap_ {j=1}^4I_i^j= 1$.
 \end{proof}

At this stage we have constructed the extraspecial group of order $2^9$ on which $F_i$ acts.

\begin{lemma} \label{celts}  The following hold:
\begin{enumerate}
\item  $C_{\Sigma_1}(Z)=R_1$, $C_{\Sigma_1}(I_1^j \cap I_1^k) =\Sigma_1^j\Sigma_1^k$ and, if $\langle x \rangle \in \mathcal P(I_1)$, then $C_{\Sigma_1}(x) =\langle r_1\rangle$.
\item  $C_{\Sigma_2}(Z)=R_2$, $C_{\Sigma_2}(I_2^j \cap I_2^k) =\Sigma_2^j\Sigma_2^k$ and, if $\langle x \rangle \in \mathcal M(I_2)$, then $C_{\Sigma_2}(x) = \langle r_2\rangle$.
\end{enumerate}
\end{lemma}

\begin{proof} We prove (i) the proof of (ii) being the same.  Let $1 \le j \le 4$.
We know that $\Sigma_1= \Sigma_1^1\Sigma_1^2\Sigma_1^3\Sigma_1^4$. Since $I_1$ acts faithfully on $\Sigma_1$, we have that $C_{I_1}(\Sigma_1^j) = I_1^j$. Thus the elements of $\mathcal P(I_1)$ act non-trivially on each $\Sigma_1^j$ and so $C_{\Sigma_1}(x) =\langle r_1\rangle$ for $\langle x \rangle \in \mathcal P(I_1)$.  Since we know that $Z$ centralizes exactly $R_1= \Sigma_1^1$ on $\Sigma_1$ we now have that (i) holds.
\end{proof}

\section{The structure of $C_G(\rho_1)$}

We continue to use our standard notation.  In this section  we are going to show that $C_G(\rho_1)$  is isomorphic to the corresponding centralizer in $\F_4(2)$. So our aim is to show that  $C_{G}(\rho_1) \cong 3 \times \Sp_6(2)$.  By   Lemma~\ref{eitheror}
we have that $C_G(\rho_1)$ either is as in $\F_4(2)$ or is isomorphic to $3 \times \Aut(\SU_4(2))$. We will show  the latter case yields  a contradiction.

\begin{lemma}\label{uniqueAisig1}
Suppose that $C_G(\rho_i) \cong 3 \times \Aut(\SU_4(2))$. Then $\Sigma_i$ is the unique maximal signalizer for
$I_i^1$ in $K_i$.
\end{lemma}

\begin{proof} We simplify our notation by assuming that $i=1$. The argument for $i=2$ is the same.
Notice that  $$\{I_1^1 \cap I_1^j\mid 2\le j\le 4\}= \mathcal M(I_i^1) .$$
The only other proper subgroup of $I_1^1$ is $Z$ by Lemma~\ref{IijIik}.  Hence, as $E_1^{1,j} \cong \SU_4(2)$ by assumption,  Lemma~\ref{Sigmapieces} (iii)(b) implies that $$\Sigma_1 \ge  O_2(C_{K_1}(I_1^k\cap I_1^j))= O_{3'}(C_{K_1}(I_1^k\cap I_1^j)).$$

 Suppose that $\Theta$ is a signalizer for $I_1^1$.
 Then $$\Theta= \langle C_\Theta(a)\mid a\in
I_1^{1\#}\rangle.$$
However,  $$C_\Theta (Z)\le O_{3'}(M\cap K_1)= R_1\le \Sigma_1$$
and, for $1<j \le 4$, by Lemma~\ref{cri} and Lemma~\ref{celts}
$$C_{\Theta}(I_1^1\cap I_1^j) \le O_{3'}(C_{K_i}(I_1^1\cap I_1^j))= \Sigma_1^1\Sigma_1^j\le \Sigma_1.$$ Hence $\Theta\le \Sigma_1$.
\end{proof}

The next lemma puts us firmly on the track of $\F_4(2)$ and $\Aut(\F_4(2))$.

\begin{lemma}\label{ItsSp62} We have $C_G(\rho_1) \cong C_G(\rho_2) \cong 3 \times \Sp_6(2)$.
\end{lemma}

\begin{proof} Suppose that the lemma is false. Then by Lemmas~\ref{eitheror} and \ref{thesame}
$$C_G(\rho_1) \cong C_G(\rho_2) \cong 3 \times \Aut(\SU_4(2)).$$
We claim that, for $i=1,2$, $\Sigma_i$ is
self-centralizing in $K_i$. Let $W_i= C_G(\Sigma_i)$. Then $W_i \le K_i$ and, as $C_S(r_i) \in \Syl_3(K_i)$ by Lemma~\ref{CSr} and since this group acts faithfully on $\Sigma_i$ by Lemma~\ref{sigmai}, we have that $W_i$ is a $3'$-group which is normalized by $I_i^1$.  By Lemma~\ref{uniqueAisig1}, $\Sigma_i$ is the
unique  maximal signalizer for $I_i^1$ and hence $\Sigma_i \ge W_i$.

Since $\Sigma_i$ is the unique maximal
signalizer for $I_i^1$ in $K_i$ it is also the unique maximal signalizer of $Q_{3-i}\ge I_i^1$ and $I_i\ge I_i^1$ in $K_i$. It
follows that $N_G(\Sigma_i) \ge \langle F_i, C_M(r_i)\rangle $ as $Q_{3-i} $ is a normal subgroup of $C_M(r_i)$.
Now $$C_M(r_i)\Sigma_i/\Sigma_i = I_iQ_{3-i}R_{3-i}\langle f \rangle \Sigma_i/\Sigma_i$$  as $R_i \le \Sigma_i$. We now deduce  $C_{C_M(Z)}(r_i)\Sigma_i/\Sigma_i $ is isomorphic to a $3$-centralizer in $\PSp_4(3)$. Furthermore, as
$\Sigma_i$ is the unique maximal signalizer for $I_i$ in $K_i$, we have that $I_i$ does not normalize any
non-trivial $3'$-subgroup of $N_G(\Sigma_i)/\Sigma_i$ and $f$ inverts $Z$. Therefore, since $F_i \le N_G(\Sigma_i)$,  Prince's
Theorem~\ref{PrinceThm} yields  $$N_G(\Sigma_i)/\Sigma_i \cong \Aut(\SU_4(2))\text{ or } \Sp_6(2).$$
Observe that $N_G(\Sigma_i) \ge \langle F_i,C_M(r_i)\rangle \ge E(C_G(\rho_{i}))$.

We claim  $N_G(\Sigma_i) = K_i$.  To prove this we intend to apply Theorem~\ref{closed} to $K_i/\langle r_i\rangle$.  We have already verified
hypotheses (i) and (ii) of that theorem.

As $N_G(\Sigma_i)/\Sigma_i \cong \Aut(\SU_4(2))\text{ or } \Sp_6(2)$, every element of $C_S(r_i)\Sigma_i/\Sigma_i$ is $N_G(\Sigma_i)/\Sigma_i$-conjugate to
an element of $I_i \Sigma_i/\Sigma_i= J(C_S(r_i))\Sigma_i/\Sigma_i$ the Thompson subgroup of $C_S(r_i)\Sigma_i/\Sigma_i$. Since $F_i$ controls fusion in $I_i$ by Lemma~\ref{fusion1}, we also have
hypothesis (iii) of Theorem~\ref{closed}.

Again to simplify notation, assume that $i=1$.
Suppose that $d $ is an element of order $3$ with $ d\in N_G(\Sigma_1)  \cap N_G(\Sigma_1)^h $ for some $h \in K_1$ is such that
$C_{\Sigma_1}(d) \neq \langle r_1\rangle$.  Then, by Lemma~\ref{celts} (i), we may suppose that $\langle d \rangle = Z$ or $\langle d \rangle = I_1^1\cap I_1^2= \langle \rho_2\rangle$. Then, as $N_{K_1}(Z)=C_{M}(r_1) \le N_G(\Sigma_1)$ and $C_{K_1}(\rho_2)=C_{C_G(\rho_{2})}(r_1)
\le N_G(\Sigma_1)$, we deduce   $$C_{K_1}(d) \le N_G(\Sigma_1).$$ On the other hand,
$C_{{N_G(\Sigma_1)}^h}(d)$ contains a $K_1$-conjugate $X$ of $I_1$. Since $X \le C_{K_1}(d) \le N_G(\Sigma_1)$, we may suppose that ${N_G(\Sigma_1)}
\cap {N_G(\Sigma_1)}^h \ge I_1$.
But then $\Sigma_1 = \Sigma_1^h$ and $N_G(\Sigma_1)=
{N_G(\Sigma_1)}^h$ as $\Sigma_1$ is the unique maximal signalizer for $I_1$ in $K_1$ by Lemma~\ref{uniqueAisig1}.    Thus the hypothesis of Theorem~\ref{closed} fulfilled and therefore
$K_1=N_G(\Sigma_1)$.

Suppose that $N_G(\Sigma_1)/\Sigma_1\cong \Aut(\SU_4(2))$.  Let $\wt \rho_1 \in \mathcal P(I_1)$. Then, as $|\mathcal P(I_1)|=3$,
$$C_{N_G(\Sigma_1)/\Sigma_1} (\wt {\rho_1}\Sigma_1) \cong 3^3.\Dih(8) $$ by Lemma~\ref{type1} (v). On the other hand, by Lemma~\ref{crossover} this group is non-soluble which is a contradiction. We
conclude that $N_G(\Sigma_1)/\Sigma_1 \cong \Sp_6(2)$. Repeating the arguments for $N_G(\Sigma_2)$ yields
$N_G(\Sigma_2)/\Sigma_2 \cong \Sp_6(2)$. Furthermore, the elements from $\mathcal P(I_1)$ act fixed point freely on $\Sigma_1/\langle r_1\rangle$ and the elements of $\mathcal M(I_2)$ act fixed point freely on $\Sigma_2/\langle r_2\rangle$,  in both cases, $i=1,2$, $\Sigma_i/\langle r_i\rangle$ is the spin module for
$N_G(\Sigma_i)/\Sigma_i$.

Since $r_2$ commutes with $I_1\cap I_2 \le N_G(\Sigma_1)$ which has Type N+ by Lemma~\ref{I1capI2},
Table~\ref{Table1} indicates that $r_2$ acts as a unitary transvection on $\Sigma_1/\langle r_1\rangle$.
Therefore $|C_{\Sigma_1/\langle r_1\rangle }(r_2)|= 2^6$ and
$$2^6 \le |C_{\Sigma_1}(r_2)|\le 2^7.$$ Since $\langle r_1, r_2 \rangle$ is centralized by $I_1\cap I_2$,
$C_{\Sigma_1}(r_2)$ is $(I_1\cap I_2)$-invariant. Because  the elements of $\mathcal P(I_1\cap I_2)$ act fixed
point freely on $\Sigma_1/\langle r_1\rangle$ (see Lemma~\ref{sp62spin}) we infer that
$|C_{\Sigma_1}(r_2)|=2^7$. Now, as $K_i=N_G(\Sigma_i)$ for $i=1,2$,  $C_{\Sigma_1}(r_2)$ normalizes
$C_{\Sigma_2}(r_1)$ and vice versa, and so
$$[C_{\Sigma_1}(r_2), C_{\Sigma_2}(r_1)] \le \Sigma_1 \cap \Sigma_2.$$
Since $r_1 \not \in \Sigma_2$ and $r_2 \not \in \Sigma_1$, $\Sigma_1 \cap \Sigma_2$ is abelian and is
centralized by $C_{\Sigma_1}(r_2)C_{\Sigma_2}(r_1)$. In particular, $\Sigma_1 \cap \Sigma_2 \le Z(C_{\Sigma_1}(r_2))$. Thus, as
 $|C_{\Sigma_1}(r_2)|= 2^7$ and $\Sigma_1$ is
extraspecial it follows that  $\Sigma_1 \cap \Sigma_2$ has order at most $2^2$ as $r_1 \not \in \Sigma_2$. We have that $I_1 \cap I_2$ acts
on $\Sigma_1 \cap \Sigma_2$. Since $|I_1 \cap I_2|= 3^2$, there is $w \in C_{I_1\cap I_2}(\Sigma_1\cap
\Sigma_2)^\#$. Now $(\Sigma_1 \cap \Sigma_2)\langle r_1 \rangle$ is elementary abelian. Since  for $a \in \mathcal
S(I_1\cap I_2)$ have $C_{\Sigma_1}(a) \cong \Q_8$ and $a \in \mathcal P(I_1 \cap I_2)$ have $C_{\Sigma_1}(a)=
\langle r_1 \rangle$, we must have $\langle w \rangle \in \mathcal M(I_1\cap I_2)$. But then $\Sigma_1\cap \Sigma_2
\le C_{\Sigma_2}(w)=1$ by Lemma~\ref{celts}. This means that $\Sigma_1 \cap \Sigma_2=1$ which then forces  $[C_{\Sigma_1}(r_2),
C_{\Sigma_2}(r_1)]=1 $ and Lemma~\ref{sp62facts} (iv) provides a contradiction.\end{proof}

\section{Some subgroups in  the centralizer of the involutions $r_1$ and $r_2$}

In this section, we finally  construct $O_2(K_i)$ where $K_i= C_G(r_i)$. Recall from Definition~\ref{F4cent},  we expect $O_2(K_i) $ to be a product of an elementary abelian group of order $2^7$ by an extraspecial group of order $2^9$. We have already located the extraspecial group $\Sigma_i$. In this section we uncover the elementary abelian group. We consider the situation for $K_1$. In the previous section we proved that $C_G(\rho_2) \cong  3 \times \Sp_6(2)$. With this additional information  we  study $C_{K_1}(\rho_2)$. This group has shape $3 \times 2^{1+2+4}.(\Sym(3) \times \Sym(3))$. For us it is important that $Z(O_2(C_{K_1}(\rho_2) ))$ is elementary abelian of order 8. Furthermore $I_1=C_J(r_1)$ normalizes this group. This time there are six conjugates of this group under the action $C_L(r_1)$ and we define a group $\Upsilon_1$ generated by these six conjugates. We  show that $\Upsilon_1$ is elementary abelian of order $2^7$ and centralizes $\Sigma_1$, the extraspecial group found earlier. Hence the product of both gives a $2$-group $\Gamma_1$ of order $2^{15}$, which is in fact isomorphic to the corresponding group in $\F_4(2)$.  Furthermore we show that $N_G(\Gamma_1)/\Gamma_1 \cong \Sp_6(2)$ and so $N_G(\Gamma_1)$  is similar to a $2$-centralizer in $\F_4(2)$. In the next section we show $K_1 = N_G(\Gamma_1)$.

We use our, by now, standard  notation.  In particular recall the definition of $\Sigma_i$ from
\ref{sigma} and $I_i^j$ the conjugates of $A_{3-i}$ under $F_i=C_L(r_i)$. Our first goal is to construct a signalizer
for $I_i^1$, $i=1,2$, which contains $\Sigma_i$ properly.  So, for $1 \le j<k\le 4$, we define
$$\Theta_i^{j,k}= Z(O_2(C_{K_i}(I_i^j \cap I_i^k)))$$ and put $$\Upsilon_i = \langle \Theta_{i}^{j,k}\mid  1 \le
j<k\le 4\rangle.$$ We will shortly show that $\Upsilon_i$ is elementary abelian group of order $2^7$.

As $C_G(I_i^j\cap I_i^k) \cong 3 \times \Sp_6(2)$, Lemma~\ref{cri} yields
$$C_{K_i}(I_i^j\cap I_i^k)\approx 2^{1+2+4}.(\Sym(3) \times \Sym(3)).$$
Hence, by Lemmas~\ref{sp62line} (iii) and (iv) and \ref{Sigmapieces}(iii), $\Theta_{i}^{j,k} $ is elementary abelian of order $2^3$ and $$O_2(C_{K_i}(I_i^j\cap I_i^k))= \Sigma_i^{j}\Sigma_i^k\Theta_{i}^{j,k}.$$ We record this latter equality.

\begin{lemma}\label{O2} For $i= 1, 2$ and $1 \le j < k \le 4$, $O_2(C_{K_i}(I_i^j\cap I_i^k))= \Sigma_i^{j}\Sigma_i^k\Theta_{i}^{j,k}.$ 
\end{lemma}

\begin{lemma} \label{UPSstruct}Suppose that $i=1,2$ and $  \{j, k,l,m\} =\{1,2,3,4\}$. Then \begin{enumerate}
\item $\Theta_i^{j,k}$ is elementary abelian of order $2^3$, contains $r_i$ and a $G$-conjugate $s_{3-i}$ of
$r_{3-i}$ with $s_{3-i}\neq r_i$.
\item $\Theta_i^{j,k}= \Theta_i^{l,m}$.
\item $\Upsilon_i$ centralizes $\Sigma_i$.
\item $\Theta_i^{j,k} \Theta_i^{k,l}  $ is elementary abelian of order $2^5$.
\item $\Upsilon_i$ is elementary abelian of order $2^7$ and is normalized by $I_i$.\end{enumerate}
\end{lemma}

\begin{proof}  To reduce the notational complexity of our argument we present the proof for $i=1$ the proof when $i=2$ is the same but we have to be careful when following the members of $\mathcal M(J)$ and $\mathcal P(J)$ in the arguments.

 By definition
$$\Theta_1^{j,k}=Z(O_2(C_{K_1}(I_1^j \cap I_1^k))).$$ We know  $I_1^j\cap I_1^k \in \mathcal M(J)$
from Lemma~\ref{IijIik} and we know  $C_{K_1}(I_1^j\cap I_1^k)\cap E_1^{j,k}$ is a line stabiliser in the
natural symplectic representation of $E_1^{j,k} \cong \Sp_6(2)$. Thus $\Theta_1^{j,k}$ is elementary abelian of
order $2^3$ by Lemma~\ref{sp62line} and of course $\Theta_1^{j,k}$ contains $r_1$ and, by Lemma~\ref{crossover}, $r_2$ is a 2-central involution in $E_1^{jk}$ and so $\Theta_1^{j,k}$
also contains a conjugate of $r_2$. This proves (i).

Now $J \cap E_1^{j,k}$ centralizes a conjugate of $r_{2}$ and is thus
$G$-conjugate to $I_2$. It follows from Lemma~\ref{type1} that $|\mathcal S(J \cap E_1^{j,k})|=4$, $|\mathcal P(J \cap
E_1^{j,k})|= 6$ and $|\mathcal M(J \cap E_1^{j,k})|=3$. Now using Lemma~\ref{sp62line} (iv), we have
$$X_1^{j,k}=C_{I_1 \cap E_1^{j,k}}(\Theta_1^{j,k}) \in \mathcal M(I_1 \cap E_1^{j,k}).$$
Observe  $X_1^{j,k} \le I_1$ and so $X_1^{j,k}$ normalizes $\Sigma_1$.

Since $X_1^{j,k} \in \mathcal
M(I_1)$,  $C_{\Sigma_1} (X_1^{j,k})$ has order $2^5$ by Lemma~\ref{celts}. As
$[\Sigma_1^j\Sigma_1^k,X_1^{j,k} ]= \Sigma_1^j\Sigma_1^k$  and $\Sigma_1$ is extraspecial, we deduce  $$C_{\Sigma_1}(X_1^{j,k}) =
\Sigma_1^l\Sigma_1^m = C_{\Sigma_1}(\Sigma_1^j\Sigma_1^k).$$  In particular, we now have  $X_1^{j,k} = I_1^l
\cap I_1^m$ by Lemma~\ref{celts}. This implies  $\Theta_1^{j,k} \le C_G( I_1^l \cap I_1^m)$ and $\Theta_1^{j,k}$ is normalized by $I_1$;
therefore
$$\langle\Theta_1^{j,k},\Sigma_1^l\Sigma_1^m\rangle = O_2(C_{K_1}(I_1^l\cap I_1^m)).$$ Since $\Theta_1^{j,k}$
is $I_1$-invariant and elementary abelian, we infer $\Theta_1^{j,k}= \Theta_1^{l,m}$ and that
$\Theta_1^{j,k}$ commutes with $\Sigma_1^j\Sigma_1^k$ as well as with $\Sigma_1^l\Sigma_1^m$. Since
$\Sigma_1=\Sigma_1^j\Sigma_1^k\Sigma_1^l\Sigma_1^m$, we  have now proved claims (ii) and (iii).

Because  $\Theta_1^{j,k} = \Theta_1^{l,m}$ we have that $\Theta_1^{j,k}$ is centralized by $\langle X_1^{j,k}, X_1^{l,m}\rangle = \langle I_1^i\cap I_1^j,I_1^l\cap I_1^m\rangle$ which has Type N- as $\Theta_1^{j,k}$ does not commute with a conjugate of $Z$. Hence $\langle \Theta_1^{j,k} , \Theta_1^{k,l}\rangle$ is centralized by
$$Y=\langle X_1^{j,k}, X_1^{l,m}\rangle \cap \langle X_1^{k,l}, X_1^{j,m}\rangle \in \mathcal P(J).$$

Now $C_G(Y) \cong 3 \times \Sp_6(2)$ and  $I_1 \cap E(C_G(Y))$ is of Type N- by Lemma~\ref{crossover}. Since
$\langle \Theta_1^{j,k} , \Theta_1^{k,l}\rangle$ centralizes $r_1$ and   is normalized by $I_1$ we infer that $r_1$
is an involution of $E(C_G(Y))$ with centralizer of shape $2^{5}.\Sp_4(2)$ and that $\langle \Theta_1^{j,k} ,
\Theta_1^{k,l}\rangle \le O_2(C_{E(C_G(Y))}(r_1))$ which is elementary abelian. But then $$\langle \Theta_1^{j,k} , \Theta_1^{k,l}\rangle =
\Theta_1^{j,k}  \Theta_1^{k,l}$$ is elementary abelian of order at most $2^5$. It now follows that
$\Upsilon_1= \Theta_1^{1,2}\Theta_1^{2,3}\Theta_1^{2,4}$ has order at most $2^7$ and is $I_1$-invariant. We have seen that $C_{I_1}(\Theta_1^{j,k}
\Theta_1^{k,l})$ is $I_1^j\cap I_1^k$. Thus $C_{I_1}(\Upsilon_i) \le I_1^1\cap I_1^2\cap I_1^3\cap I_1^4=1$.
Hence $I_1$ acts faithfully on $\Upsilon_1$ and so $|\Upsilon_1|=2^7$. This completes the proof of (iv) and  (v) and the
verification of the  statements in the lemma.
\end{proof}

For $i=1,2$, we now set  $$\Gamma_i = \Sigma_i\Upsilon_i.$$

\begin{lemma}\label{Gammabasic} For $i=1,2$, we have that $\Gamma_i$ has order $2^{15}$ and is normalized by $F_i$. Furthermore the following hold.
\begin{enumerate}
\item $Z(\Gamma_i) = \Upsilon_i$; and
\item $[\Gamma_i,\Gamma_i]= \langle r_i\rangle$.
\end{enumerate}
\end{lemma}

\begin{proof}  By Lemmas~\ref{sigmai} and \ref{UPSstruct},   $\Sigma_i$ has order $2^9$ and is extraspecial and $|\Upsilon_i|= 2^7$ and centralizes $\Sigma_i$. This yields   $\Upsilon_i \cap \Sigma_i = \langle r_i\rangle$ and  $\Gamma_i$ has order $2^{15}$. Furthermore,  as $\Sigma_i$ is extraspecial, $\Upsilon_i$ is elementary abelian and $\Upsilon_i$ commutes with $\Sigma_i$ we have that $\Upsilon_i= Z(\Gamma_i)$ and
$[\Gamma_i,\Gamma_i]= \langle r_i\rangle$. Hence (i) holds.

By the construction of $\Sigma_i$ and $\Upsilon_i$, $F_i$ normalizes both groups and consequently also normalizes their product. This is (ii).
\end{proof}

\begin{lemma}\label{MaxSig} For $i=1,2$, $\Gamma_i$  is the unique maximal signalizer for $I_i^1$ in $K_i$.
\end{lemma}
\begin{proof} Assume that $W$ is an $I_i^1$ signalizer in $K_i$. Then $$W= \langle C_W(x) \mid x \in (I_i^1)^\#\rangle.$$
If $\langle x \rangle =Z\in \mathcal S (I_i^1)$, then $$O_{3'}(C_{K_i}(Z))=R_i =\Sigma_i^1 \le \Sigma_i\le
\Gamma_i$$ is the unique maximal $I_i^1$ signalizer in $C_{K_i}(Z)$. All the other subgroups of order $3$ in
$I_i^1$ are conjugate to $\langle \rho_{3-i}\rangle$ by an element of $Q_{3-i} \le F_i$. Hence we only need to
consider $I_i^1$ signalizers in $C_{K_i}(\rho_{3-i})$.
By Lemma \ref{ItsSp62}, $C_{G}(\rho_{3-i}) = C_G(I_i^1\cap I_i^2) \cong 3 \times \Sp_6(2)$ and we know from Lemma~\ref{cri} that $$C_{K_i}(\rho_{3-i}) \approx 3 \times 2^{1+2+4}.(\Sym(3) \times \Sym(3)).$$ Set $D= C_{K_i}(\rho_{3-i})$. Then $$O_2(D)=  \Sigma_i^1\Sigma_i^2\Theta_i^{1,2} \le \Gamma_i$$ and, Lemma~\ref{sp62line}(ii), implies   $ZO_2(D)/O_2(D)$ is diagonal in $D/O_2(D)$. Since
$C_W(\rho_{3-i})$ is normalized by $Z$ we infer that $C_W(\rho_{3-i}) \le \Gamma_i$ as claimed.
\end{proof}

\begin{lemma}\label{fusionr_i} For $i=1,2$, there is a $G$-conjugate of $r_i$ in $\Gamma_i\setminus \Upsilon_i$.
\end{lemma}

\begin{proof} This fusion can already be seen in $$C_{K_i}(\rho_{3-i})\approx 3 \times 2^{1+2+4}.(\Sym(3)\times\Sym(3))$$
as $r_i$ is not
weakly closed in $O_2(C_{K_i}(\rho_{3-i}))$ with respect to $C_G(\rho_{3-i})$  by Lemma~\ref{sp62line} (vi).
\end{proof}

We are now able to determine the structure of $N_G(\Gamma_i)$.

\begin{lemma}\label{itssp} For $i=1,2$, the following hold.
\begin{enumerate}
\item $N_G(\Gamma_i)/\Gamma_i \cong \Sp_6(2)$;
\item as $N_G(\Gamma_i)/\Gamma_i$-modules,  $\Gamma_i/\Upsilon_i$ is a spin module and $\Upsilon_i/\langle r_i \rangle$ is a natural module;
\item $\syl_2(N_G(\Gamma_i)) \subseteq \syl_2(K_i)$; and
\item if $T \in \syl_2(N_G(\Gamma_i))$, then  $\Gamma_i/\langle r_i \rangle = J(T/\langle r_i \rangle)$, $Z(T) \le \Upsilon_i$ and $Z(T)$ has order $4$.
\end{enumerate}
In particular, $N_G(\Gamma_i)$ is similar to a $2$-centralizer in $\F_4(2)$.
\end{lemma}

\begin{proof}
We already know that $\Gamma_i$ is normalized by $F_i$ and we have that $\Gamma_i$ is the unique maximal
$I_i^1$-signalizer in $K_i$ by Lemma~\ref{MaxSig}. It follows that $\Gamma_i$ is also the unique maximal signalizer for $Q_{3-i} \ge
I_i^1$ in $K_i$. Therefore $N_{E (C_G(\rho_i))}(Q_{3-i})$ also normalizes $\Gamma_i$. It follows from
\cite[page 46]{Atlas} that  $$X=\langle F_i, N_{E (C_G(\rho_i))}(Q_{3-i})\rangle\cong \Aut(\SU_4(2))$$  and
$X$ normalizes $\Gamma_i$.

 Since $C_{K_i}(Z)\Gamma_i/\Gamma_i$ is a $3$-centralizer of type $\PSp_4(3)$, $\Gamma_i$ is a maximal signalizer for $I_i^1$ and $Z$ is inverted in $N_G(\Gamma_i)/\Gamma_i$, we deduce
 $N_G(\Gamma_i)/\Gamma_i \cong \Sp_6(2)$ or $\Aut(\SU_4(2))$  by using Theorem~\ref{PrinceThm}.

We know that $I_i$ acts faithfully on both $\Gamma_i/\Upsilon_i$  and $\Upsilon_i/\langle r_i \rangle$.
In particular, as  $|\Upsilon_i/\langle r_i\rangle|= 2^6$,  if $N_G(\Gamma_i)/\Gamma_i\cong \Aut(\SU_4(2))$  then $\Upsilon_i/\langle r_i\rangle$ is an orthogonal module and if $N_G(\Gamma_i)/\Gamma_i\cong \Sp_6(2)$  then $\Upsilon_i/\langle r_i\rangle$ is a natural module. Similarly since $C_{\Sigma_i}(Z)= \Sigma_i^1$ and since this  subgroup is not normalized by $F_i$ and $|\Gamma_i/\Upsilon_i|=2^8$, if  $N_G(\Gamma_i)/\Gamma_i\cong \Aut(\SU_4(2))$,  then $\Gamma_i/\Upsilon_i$ is an natural module and,  if $N_G(\Gamma_i)/\Gamma_i\cong \Sp_6(2)$,  then $\Gamma_i/\Upsilon_i$ is a spin module (see Lemma~\ref{modfacts}). So once we have proved part (i), part (ii) will also be proved.

Next we prove (iii) and the first part of (iv). Let $T \in \syl_2(N_G(\Gamma_i))$. Since, by Lemma~\ref{NotF},  $\Gamma_i/\langle r_i\rangle $ is not an $F$-module for $N_G(\Gamma_i)/\Gamma_i$, \cite[Lemma~26.15]{GLS2} implies that $ \Gamma_i/\langle r_i\rangle$ is the Thompson subgroup of $T/\langle r_i\rangle$. It follows that $N_{K_i}(T) \le N_G(\Gamma_i)$ and,
in particular,  $T \in \Syl_2(K_i)$ and $N_{K_i}(T)=T$. Notice furthermore that $N_G(\Gamma_i)/\langle r_i\rangle$ controls $K_i/\langle r_i\rangle$-fusion  in $\Gamma_i /\langle r_i\rangle$.  The last two parts of (iv) follow from the fact that $\Sigma_i$ is extraspecial and Lemma~\ref{nonsplitmods}.

It remains to prove (i). Assume that $N_G(\Gamma_i)/\Gamma_i \cong\Aut(\SU_4(2))$.
Using Lemma~\ref{fusionr_i}, there exists $g\in G$ and $s \in \Gamma_i\setminus \Upsilon_i$ such that $s=r_i^g$.   Since $N_G(\Gamma_i^g)$ contains a Sylow $2$-subgroup of $C_G(s)$,  there is a $h \in C_G(s)$ such that $C_{\Gamma_1}(s)^h \le N_G(\Gamma_i^g)$ and we have $s= r_i^{gh}$ so we may suppose  $g$ was chosen so  $C_{\Gamma_1}(s)\le N_G(\Gamma_i^g)$. Note that, as $s \in \Gamma_i\setminus \Upsilon_i$,  $s$ is conjugate in $\Gamma_i$ to $sr_i$ and, as $N_G(\Gamma_i)/\langle r_i\rangle$ controls $K_i/\langle r_i\rangle$-fusion  in $\Gamma_i /\langle r_i\rangle$, $s$ is not $K_i$-conjugate to an element of $\Upsilon_i$.

Since $C_{\Gamma_1}(s)$ contains an extraspecial group of order $2^7$ with derived group $\langle r_i\rangle$, and  $\Aut(\SU_4(2))$  does not (by Lemma~\ref{sp62facts}),  we have  $r_i \in \Gamma_i^g$. It follows that $C_{\Gamma_i^g}(r_i)$, which has index at most $2$ in $\Gamma_i$, also contains an extraspecial group of order $2^7$. As $T \in \Syl_2(K_i)$,  there is $f \in K_i$ such that $C_{\Gamma_i^g}(r_i)^f=C_{\Gamma_i^{gf}}(r_i)\le T$.  It follows that $s^f \in \Gamma_i \setminus \Upsilon_i$
and we may as well suppose that $s=s^f$ (though we may no longer have $C_{\Gamma_1}(s)\le N_G(\Gamma_i^g)$).
With this choice of $s$,    $|\Gamma_i^g : \Gamma_i^g\cap N_G(\Gamma_i)| \le 2$.
Now $$\Phi(\Gamma_i^g \cap \Gamma_i) \le \Phi(\Gamma_i^g) \cap \Phi(\Gamma_i)= \langle s\rangle \cap \langle r_i\rangle=1$$ which means  $\Gamma_i^g \cap \Gamma_i$  is elementary abelian. As $\Gamma_i$ contains $\Sigma_i$ which is extraspecial of order $2^9$, this  yields   $|\Gamma_i^g\cap \Gamma_i| \le 2^{11}$ and so $$|({\Gamma_i^g}\cap N_G(\Gamma_i))\Gamma_i/\Gamma_i| \ge 2^3.$$
Further
 $$[\Upsilon_i \cap \Gamma_i^g, N_G(\Gamma_i) \cap  \Gamma_i^g] \leq [\Gamma_i^g, \Gamma_i^g]\cap \Upsilon_i =\langle s \rangle \cap \Upsilon_i = 1.$$
Hence, as  $|({\Gamma_i^g}\cap N_G(\Gamma_i))\Gamma_i/\Gamma_i| \ge 2^3$, Lemma~\ref{sp62facts}(iii)  (which says that $\Aut(\SU_4(2))$ contains no fours group of unitary transvections) implies  $|\Upsilon_1 \cap \Gamma_i^g| \leq 2^5$. Therefore  $|\Gamma_i \cap \Gamma_i^g| \leq 2^9$. We have now shown  $|(\Gamma_i^g \cap N_G(\Gamma_i))\Gamma_i/\Gamma_i| \geq 2^5$ which, as this group is elementary abelian and the $2$-rank of $\Aut(\SU_4(2))$ is $4$,  is contradiction. Therefore $N_G(\Gamma_i)/\Gamma_i \cong \Sp_6(2)$ and this completes the proof of part (i) and thereby also (ii).

\end{proof}

\section{The centralisers of $r_1$ and $r_2$}

In this section   we finally determine the structure of $K_i=C_G(r_i)$. We will prove  $K_i = N_G(\Gamma_i)$ and hence conclude that $K_i$ is similar to a $2$-centralizer in $\F_4(2)$. The plan is to show $\Upsilon_i$ is strongly closed in a Sylow $2$-subgroup of  $K_i$ with respect to $K_i$ and then to quote Goldschmidt's Theorem in the form of Lemma~\ref{Gold} to show that $K_i=N_G(\Gamma_i)$. To achieve  this we  study $K_i$-fusion of involutions. As most of the centralizers of involutions in $N_G(\Gamma_i)$ have order divisible by three, this will be reduced to fusion of 3-elements. Hence the first lemma we prove in this section  will be that $N_G(\Gamma_i)$ is strongly $3$-embedded in $K_i$, which means that we have control of fusion of elements of order $3$ in $K_i$.

We use all our previous  notation and furthermore for this section we set $H_i = N_G(\Gamma_i)$.

\begin{lemma}\label{Strong3} For $i=1,2$,
 $H_i$ is strongly $3$-embedded in $K_i$. In particular, $H_i$ controls fusion of elements of
order $3$ in $H_i$.
\end{lemma}

\begin{proof}

Suppose that $x \in H_i$ has order $3$. We will show  $C_{K_i}(x)$ normalizes $\Gamma_i$.  Recall $C_S(r_i) \in \Syl_3(K_i)$ and $C_S(r_i) \le F_i \le H_i$ so   $C_S(r_i)$ normalizes $\Gamma_i$. Since every element of order $3$ in $C_S(r_i)$ is $H_i$-conjugate into $I_i$, we may suppose  $x \in I_i$.

Again to simplify our notation slightly we consider the case when $i=1$. Thus $|\mathcal S(I_1)|= 4$, $|\mathcal M(I_1)|= 6$ and $|\mathcal P(I_1)|=3$ by Lemma~\ref{I1}.  If $\langle x \rangle \in \mathcal S(I_1)$, then  we may suppose that $\langle x \rangle = Z$. In this case, by Lemma~\ref{CSr}   $$C_{K_1}(Z) = Q_2R_1R_2I_1 \le H_1.$$
So suppose that $\langle x\rangle = \langle \rho_2\rangle \in \mathcal M(I_1)$. Then, by Lemma~\ref{O2},  $$C_{K_1}(\rho_2) =  \Sigma_1^1\Sigma_1^2\Theta_1^{1,2}N_{F_1}(I_1 \cap E_1^{12}) \le \Gamma_1 F_1 \le H_1.$$
Suppose $\langle x \rangle = \wt {\rho_1} \in\mathcal P(I_1)$. Then, by Lemma~\ref{crossover},  $C_{K_1}(\wt {\rho_1}) \approx 3 \times 2^5{:}\Sp_4(2)$ and this has the same order as  $C_{H_1}(\wt {\rho_1})$.  Thus $C_{K_1}(\wt {\rho_1}) \le H_1$.  Finally, $N_{K_1}(C_S(r_1)) \le N_{K_1}(Z)$ and so $H_1$ is strongly $3$-embedded in $K_1$ by \cite[Lemma 17.11]{GLS2}.
 \end{proof}

 We next show $H_i = K_i$ for  $i = 1,2$ the proof is accomplished through a series of lemmas. It suffices to prove this with $i=1$ as the proof for $i=2$ is the same. By Lemma~\ref{itssp} (ii), $Z(H_1) = \langle r_1\rangle$, $\Upsilon_1/Z(H_1)$ is the natural $\Sp_6(2)$-module and
$\Gamma_1/\Upsilon_1$ is the spin module for $\Sp_6(2)$. Let $T$ be a Sylow 2-subgroup of $H_1$. From
Lemma~\ref{itssp} (iv) we have $T \in \syl_2(K_1)$.

\begin{lemma}\label{sc} \begin{enumerate}\item  If  $x \in \Upsilon _1^\#$ and $s\in x^{K_1}$, then $s$ and $sr_1$ are not $K_1$-conjugate.\item $\Upsilon_1$ is strongly closed in $\Gamma_1$ with
respect to $K_1$. \end{enumerate}
\end{lemma}

\begin{proof}  (i) Obviously, if $x= r_1$, the result is true. So we may suppose that $x \in \Upsilon_1 \setminus \langle r_1\rangle$. Since $H_1$ acts transitively on $(\Upsilon_1/\langle r_1\rangle)^\#$, we may additionally assume
$ x\langle r_1\rangle \in C_{{\Upsilon_1/\langle r_1\rangle}}( T)$ which has order $2$ by
Lemma~\ref{sp62natural}. As by Lemma~\ref{nonsplitmods} the preimage of $ C_{{\Upsilon_1/\langle r_1\rangle}}( T)$ is centralized by $T$ we
have $x \in Z(T)$.

Suppose that $x$ is $K_1$-conjugate to $xr_1$. Then as $x$ and $xr_1 \in Z(T)$, this conjugation must happen in
$N_{K_1}(T)$. Since $T \in \syl_2(K_1)$, this is impossible and it follows that $x$ is not $K_1$-conjugate to
$xr_1$. This proves (i)

Now consider  $y \in \Gamma_1\setminus \Upsilon_1$.  Then $[y,\Gamma_1] =\langle r_1\rangle$  and so $y$ is
conjugate to $r_1y$ in $\Gamma_1$. Therefore (i) implies (ii).
\end{proof}

\begin{lemma}\label{snormsigma} Let $x \in \Upsilon_1$, $g \in K_1$ and assume that $s=x^g$ with $s \in T \setminus \Gamma_1$. Then $s$ normalizes a $H_1$-conjugate of $I_1\Gamma_1$ and $\Sigma_1$.
\end{lemma}

\begin{proof}  Since in $H_1/\Gamma_1\cong \Sp_6(2)$ every
involution is conjugate in to $N_{H_1/\Gamma_1}(I_1 \Gamma_1/\Gamma_1)$, we may as well suppose that $s  $
normalizes $I_1 \Gamma_1$. In particular by Lemma~\ref{sigmai}  we may additionally  assume $\Sigma_1^s =
\Sigma_1$.
\end{proof}

\begin{lemma}\label{3prime}  Let $x \in \Upsilon_1$, $g \in K_1$ and assume that $s=x^g$ with $s \in T \setminus \Gamma_1$.
 Then the following hold:
\begin{enumerate}
\item $C_{\Gamma_1/\Upsilon_1}(s) = C_{\Gamma_1}(s)\Upsilon_1/\Upsilon_1$; and
\item$C_{H_1}(s)$ is a $3^\prime$-group.
\end{enumerate}
\end{lemma}
\begin{proof} By Lemma~\ref{snormsigma} we may assume that $s$ normalizes both $I_1\Gamma_1$ and $\Sigma_1$.  Let $w \Upsilon_1 \in C_{\Gamma_1/\Upsilon_1}(s)$ and write $w= w_*u$ where $w_* \in \Sigma_1$
  and $u \in \Upsilon_1$. Then $$[w,s]= [w_*u,s] = [w_*,s][u,s] \in \Upsilon_1.$$
  As $s$ normalizes $\Sigma_1$, this
means that $[w_*,s]\in \Sigma_1\cap \Upsilon_1 = \langle r_1\rangle$. Since $x$ is not $K_1$-conjugate to $sr_1$,
we deduce that $w_*$ is centralized by $s$ and this proves (i).

Suppose that $W \in \Syl_3(C_{H_1}(s))$ and let $U \in \Syl_3(C_{H_1}(x))$. Then, as $ \Upsilon_1/\langle
r_1\rangle$ is the natural $\Sp_6(2)$-module, $U$ has order $3^2$ by Lemma~\ref{sp62natural}. Since
by Lemma~\ref{Strong3} $H_1$ is strongly $3$-embedded in $K_1$ we know that $U\in  \syl_3(C_{K_1}(x))$ and so
$U^g \in \Syl_3(C_{K_1}(s))$. Thus there exists $h \in C_{K_1}(s)$ so that $U^{gh} \ge W$. Consequently  $W \le
H_1 \cap H_1^{gh}$ and Lemma~\ref{Strong3}  yields $gh \in H_1$ which contradicts the fact that $s = x^{gh}$, $s
\in T \setminus \Sigma_1\Upsilon_1$ and $x \in \Upsilon_1$.
\end{proof}

Suppose that $s^* \in s\Gamma_1$ is an involution  which is conjugate to $s$ in $K_1$.

Then $ws= s^*$ with $w \in \Gamma_1$. We claim that $w \in C_{\Gamma_1}(s)$.  To see this we note that the other possibility is that $w^{s}= w^{-1}= wr_1$ and then  we  calculate
$${s^*}^{s}= (ws)^{s}= w^{s}s= w^{-1}s=wr_1s= s^*r_1$$ which contradicts Lemma~\ref{sc}(i).

Let
$q \in C_{\Gamma_1}(s)$ and assume that $[w,q]\neq 1$. Then, by Lemma~\ref{Gammabasic}, $[w,q]=r_1$ and
$$s^*{^q}=(ws)^q= w^qs= w[w,q]s= wsr_1=s^*r_1$$, which is also impossible. Therefore $w \in Z(C_{\Gamma_1}(s))$. Since $s$ normalizes $\Sigma_1$ and $\Sigma_1$ is extraspecial, the Three Subgroup Lemma implies $Z(C_{\Sigma_1}(s))= [\Sigma_1,s]$.  Thus
 Lemma~ \ref{sc}(i) implies that

\begin{lemma}\label{newclaim}   Let $x \in \Upsilon_1$, $g \in K_1$ and assume that $s=x^g$ with $s \in T \setminus \Gamma_1$. If $s$ is $H_1$-conjugate to $s^*=w s$ where $w \in \Gamma_1$, then
$w \in Z(C_{\Gamma_1}(s))\le  [\Gamma_1,s]\Upsilon_1$. In particular, $s\Upsilon_1$ is $\Gamma_1/\Upsilon_1$-conjugate to
$s^*\Upsilon_1$ and $C_{H_1/\Gamma_1}(s\Upsilon_1)= C_{H_1/\Upsilon_1}(s)\Gamma_1/\Gamma_1$.
\end{lemma}

Now we are going to identify the involution $s\Gamma_1$ in $H_1/\Gamma_1 \cong \Sp_6(2)$.

\begin{lemma}\label{c2}  Let $x \in \Upsilon_1$, $g \in K_1$ and assume that $s=x^g$ with $s \in T \setminus \Gamma_1$.  Then $s\Gamma_1$ is an
involution of type $c_2$ and all $K_1$-conjugates of $x $ in $H_1\setminus \Gamma_1$ project to elements of this type.
\end{lemma}

\begin{proof}

By  Lemma~\ref{sp62facts} (i),  $s \Gamma_1$ is an involution  of type $a_2$, $b_1$, $b_3$ or $c_2$  in
$H_1/\Gamma_1\cong\Sp_6(2)$.  If $s\Gamma_1$ is of type $b_3$, then Lemma~\ref{sp62facts} implies that $[\Gamma_1/\langle r_1\rangle,s] = C_{\Gamma_1/\langle r_1\rangle}(s)$ and consequently $3$ divides $|C_{H_1}(s)|$. Hence  $s\Gamma_1$ is not of type $b_3$ by Lemma~\ref{3prime} (ii).

%

If $s\Gamma_1$ is of type $b_1$ or $a_2$, then, by Lemma~\ref{newclaim},  $|C_{H_1/\Upsilon_1}(s)|$ is divisible by $3^2$. If $s\Gamma_1$ is of type $a_2$, then  Lemma~\ref{sp62facts} implies $$|C_{\Upsilon/\langle r_1 \rangle}(s)/[\Upsilon/\langle r_1 \rangle,s]| = 4$$ and so $s$ is centralized by an element of order 3 contrary to Lemma~\ref{3prime} (ii). Thus $s\Gamma_1$ is not of type $a_2$. If $s\Gamma_1$ is of type $b_1$, then  Lemma~\ref{sp62facts} yields  $C_{\Upsilon/\langle r_1 \rangle}(s)/[\Upsilon/\langle r_1 \rangle,s]$ is the natural $\Sp_4(2)$-module and, as $\Sp_4(2)$ acts transitively on the non-trivial elements of this module, we again see $s$ is centralized by a $3$-element, a contradiction. Thus $s\Gamma_1$ must be of type $c_2$.
\end{proof}

\begin{lemma}\label{strongc}   $\Upsilon_1$
is strongly $2$-closed in $T$ with respect to $K_1$.
\end{lemma}

\begin{proof} Let $x \in \Upsilon_1$, $g \in K_1$ and assume that $s=x^g$ with $s \in T \setminus \Gamma_1$. By Lemma~\ref{c2}, $s$ acts as an element of type $c_2$ on the natural $\Sp_6(2)$-module.

Let $F=C_{\Sigma_1}(s)=[\Sigma_1,s]$. Then $F$ has order $2^5$ by Lemma~\ref{sp62facts}. Thus the coset $Fs$ consists solely of conjugates
of $s$ and of $sr_1$  and $F \cap \Upsilon_1 = \langle r_1\rangle$.

 Recall that we may suppose that $x \in Z(T)$. So  $s$ is a $2$-central element of $K_1$. 
Hence, as $F$ is a $2$-group which centralizes $s$, $F$ is contained in a Sylow $2$-subgroup $T_0$ of $K_1$ which
centralizes $s$. Let $\Gamma_1^*$ be the preimage of $J(T_0/\langle r_1\rangle)$, $\Upsilon_1^*=Z(\Gamma_1^*)$
and $H^*= N_G(\Gamma_1^*)$. By Lemma~\ref{itssp} we have that $\Gamma_1^*$ is conjugate to $\Gamma_1$ in $K_1$.  Then also $H^*$ is $K_1$-conjugate to $H_1$ and $H^*/\Gamma_1^*\cong \Sp_6(2)$.

Assume that $y \in F\setminus \langle r_1 \rangle$. Then $ys$ is conjugate to either $s$ or $sr_1$. In particular any coset of $\langle r_1 \rangle$ in $F$ contains some $y$ such that $ys$ is conjugate to $s$ in $K_1$. If $y \in
\Gamma_1^*$, then, as $y \in \Gamma_1\setminus \Upsilon_1$,  Lemma~ \ref{sc} (ii) yields $y \not \in \Upsilon_1^*$ and consequently we also have $ys \in \Gamma_1^*\setminus
\Upsilon_1^*$  which contradicts Lemma~\ref{sc}. Thus $y \not \in \Gamma_1^*$ and the coset $y\Gamma_1^*$ contains
$ys$. We deduce with Lemma~\ref{c2} that $y\Gamma_1^*$ is of type $c_2$ in $N_{K_1}(\Gamma_1^*)/\Gamma_1^*$ and
$F\Gamma_1^*/\Gamma_1^*$ is a  subgroup of order $2^4$ in which all the non-trivial elements are in class $c_2$.
Since $\Sp_6(2)$ has no such subgroups by Lemma~\ref{sp62facts}, we have a contradiction. Therefore  $\Upsilon_1$
is strongly 2-closed in $T$ with respect to $K_1$.
\end{proof}

Next we can prove the main result of this section:

\begin{lemma}\label{H=K} For $i=1,2$, we have $H_i= K_i$. In particular, $K_1$ and $K_2$ are similar to $2$-centralizers in $\F_4(2)$,
\end{lemma}

\begin{proof} Again it is enough to prove the lemma for $i = 1$. By Lemma~\ref{strongc} we have that  $\Upsilon_1$
is strongly 2-closed in $T$ with respect to $K_1$. Therefore Lemma~\ref{Gold} yields $K_1\le
N_{G}(\Upsilon_1)$. Now $C_{K_1}(\Upsilon_1) \cap C_S(r_1)=1$  and so $C_{K_1}(\Upsilon_1)$ is a $3'$-group.
Since, by Lemma~\ref{MaxSig}, $\Gamma_1$ is the unique maximal $I_1^1$-signalizer in $K_1$, we conclude  $\Gamma_1 \ge
C_{K_1}(\Upsilon_1)$ and thus $\Gamma_1 =C_{K_1}(\Upsilon_1)$.  It follows that $K_1=N_{K_1}(\Upsilon_1) =
N_{K_1}(\Gamma_1)$ as claimed.
\end{proof}

\section{Proof of Theorem~\ref{MT}}

Having determined the shapes of the centralizers of the involutions $r_1$ and $r_2$ in this section  we accomplish the final identification of $G$.

Let $T \in \Syl_2(K_1)$, where $K_1 = C_G(r_1)$, and recall that $\Gamma_1= \Sigma_1\Upsilon_1= O_2(K_1)$. The  conclusion of the work of the previous sections is that $K_1$ is similar to a $2$-centralizer in $\F_4(2)$.

By Lemma~\ref{UPSstruct}, $\Upsilon_1$ contains a $G$-conjugate $s_2$ of $r_2$ with $s_2\neq r_1$. As $K_1$
acts transitively on the non-trivial elements of $\Upsilon_1/\langle r_1\rangle$,   Lemma~\ref{nonsplitmods} shows that we may further suppose  that $s_2 \in
Z(T)$ and $Z(T)=\langle r_1,s_2\rangle$. Define $U_2 = C_G(s_2)$. We have $U_2$ is $G$-conjugate to $K_2= C_G(r_2)$ and
thus, as $|K_1|=|K_2|$, we have  $T \in \syl_2(U_2)$.

We will use the two groups to construct a subgroup $P = \langle K_1,U_2 \rangle \cong F_4(2)$ using Theorem~\ref{P=F4}. Recall Definition~\ref{F4setup},  and note that $K_1$, $U_2$, $T$ is an $\F_4$ set-up.

\begin{lemma}\label{P=F4new}  $P = \langle K_1, U_2\rangle \cong \F_4(2)$.
\end{lemma}

\begin{proof} This follows directly from Theorem~\ref{P=F4}.
\end{proof}

In fact we have the following corollary:

\begin{corollary}\label{F42sub} If $X$ is any group which satisfies the assumptions of Theorem~{\rm \ref{MT}}, then $X$ contains a subgroup isomorphic to $\F_4(2)$.
\end{corollary}

\begin{proof} This follows immediately from Lemma~\ref{P=F4new}. \end{proof}

Our aim is to show that $G$ is isomorphic to either $\F_4(2)$ or $\Aut(\F_4(2))$. For this we will show that $P$ is normal in $G$. As a first step we show that $P$ is normalized by $M$ and that $P_0=PM$ is either $\F_4(2)$ or $\Aut(\F_4(2))$. We then produce a normal subgroup $G_*$ of $G$ of index at most two such that $P_0 \cap G_* = P$. Our objective is then to show  $G_* = P$. This will be done using Holt's Theorem (Lemma~\ref{Holt}). Hence  we  have to gain control of $G_*$-fusion of involutions in $P$. For this we show that $P_0$ is strongly $3$-embedded in $G_*$, which will imply that $P$ controls $G_*$-fusion in $P$. We start with the proof that $M$ normalizes $P$.

We have $C_P(\rho_1) \cong C_P(\rho_2) \cong 3 \times \Sp_6(2)$ and so, by Lemma~\ref{ItsSp62},  $C_G(\rho_i) = C_P(\rho_i)$, $i = 1,2$.  As $\langle C_M(\rho_1), C_M(\rho_2)\rangle = M \cap P$, we see   $ \langle C_G(\rho_1)
,C_G(\rho_2)\rangle$ satisfies the assumptions of Theorem~\ref{MT}. By Corollary~\ref{F42sub} we get that $
\langle C_G(\rho_1) ,C_G(\rho_2)\rangle$ contains a subgroup isomorphic to $\F_4(2)$. As $P \cong \F_4(2)$, we obtain

\begin{lemma}\label{subF42} $ \langle C_G(\rho_1) ,C_G(\rho_2)\rangle = P$.
\end{lemma}

\begin{lemma}\label{PM} $M$ normalizes $P$.

\end{lemma}

\begin{proof} Since $P \cong \F_4(2)$ and $\rho_1$ and $\rho_2$ are not conjugate in $P$, we have that $M \cap P = RS\langle f \rangle$. If $M \le P$, we have nothing to do. If $M > M \cap P = RS\langle f \rangle$, then, by Lemma~\ref{structM}, there is an element $t$ of $M\setminus M \cap P$ such that $\rho_1^t = \rho_2$. This element normalizes $P$ by Lemma~\ref{subF42}. Thus $M$ normalizes $P$.
\end{proof}

Define $P_0= PM$.

\begin{lemma} \label{strong3a}  $P_0$ is strongly $3$-embedded in $G$.
\end{lemma}

\begin{proof}
Since $P \cong \F_4(2)$, there are three conjugacy classes of elements of order $3$ in $P$ and they are all witnessed
in $J$.   For $\langle x\rangle \in \mathcal S( J)$, we have $N_G(\langle x \rangle) = M \le P_0$ and for $\langle x \rangle \in \mathcal M(J) \cup \mathcal P(J)$ we have $C_G(x) = C_P(x)$ by Lemma~\ref{ItsSp62}.
 Since also $N_G(S) \le M \le P_0$ we have $P_0$ is strongly
$3$-embedded in $G$ by \cite[Lemma 17.11]{GLS2}.
\end{proof}

We can now determine the structure of $P_0$.

\begin{lemma}\label{P0syl}  We have $P_0$ contains a Sylow $2$-subgroup of $G$ and either $P_0=P$ or $P_0 \cong \Aut(\F_4(2))$.
\end{lemma}

\begin{proof}  Assume that $T \not \in \syl_2(G)$ and let $T_1 >T$ normalize $T$. Then $T_1$ normalizes  $Z(T) = \langle r_1, s_2\rangle$. Since $K_1 \le P$ and $U_2\le P$, there exists $x \in T_1$ such that $r_1^x \neq r_1$ and $s_2^x \neq s_2$. Since $Z(T)$ has order $4$, we deduce that  $r_1^x= s_2$ and thus that $K_1^x=U_2$.
Hence $x$ normalizes $P = \langle K_1, U_2 \rangle$ and $P_0=P\langle x \rangle \cong \Aut(\F_4(2))$.

Now let $T_0 \in \syl_2(P_0)$  ($P_0 = P$ or $P_0 = \Aut(P)$) and assume that $w \in N_G(T_0)$. As $r_1 \in T^\prime \leq T_0^\prime \leq T$, we have $r_1^w \in T \leq P$. Employing Lemma~\ref{F42Classes} we see that  all involutions of $P$ commute with elements of order $3$. By Lemma~\ref{strong3a}  $C_{P_0}(r_1^w)$  contains a Sylow $3$-subgroup of $C_G(r_1^w)$. Hence it follows  that $r_1^w \in r_1^{P_0}\cup s_2^{P_0}$. Then there is $x \in P_0$ such that $r_1=r_1^{wx}$ or $s_2= r_1^{wx}$. Since $\langle K_1, U_2 \rangle = P$, we have $wx \in P$.  However this means  $w \in P_0$ and we infer   $T_0\in \Syl_2(G)$.
\end{proof}

Now we produce the normal subgroup $G_*$ with $G_* \cap P_0 = P$.

\begin{lemma}\label{index2} If $P_0 > P$, then $G$ has a subgroup $G_*$ of index $2$  with  $P = P_0 \cap  G_*$.  Furthermore $G_*$ satisfies the hypothesis of Theorem~{\rm \ref{MT}}.
\end{lemma}

\begin{proof}  We let $T_0 \in \syl_2(P_0)$ and $T \in \syl_2(P)$ with $T_0>T$.   Suppose that $t \in T_0$ is an involution and $C_{P_0}(t)$ has a non-trivial Sylow $3$-subgroup $D$. Then as $P_0$ is strongly $3$-embedded by Lemma~\ref{strong3a} we have that $D \in \Syl_3(C_G(t))$.
Now by Lemma~\ref{F42Classes}   $P$ has four conjugacy classes of involutions and their centralizers have $3$-parts of their orders $3^4$, $3^4$, $3^2$ and $3^2$. On the other hand, if we let $x \in T_0\setminus T$ with  $C_{P_0}(x) \cong 2 \times {}^2\F_4(2)$, then $C_P(x)$ has Sylow $3$-subgroups which are extraspecial of order $3^3$. It follows that $x$ is not conjugate to any element in $T$ and consequently $G$ has a subgroup $G_*$ of index $2$ by Thompson's Transfer Lemma \cite[Lemma 15.16]{GLS2}. Obviously then $P_0 \cap G_* = P$ and $G_*$ satisfies the hypothesis of Theorem~\ref{MT}.
\end{proof}

We finally prove that $G \cong \F_4(2)$ or $\Aut(\F_4(2))$.

\begin{proof}[Proof of Theorem~{\rm \ref{MT}}]  By Lemma~\ref{index2}, we may suppose that $P=P_0.$ Using  Lemma~\ref{F42Classes}, $P$ has
exactly four conjugacy classes of involutions and each such  involution $t$ has $|C_P(t)|_3 \neq 1$. Since $P$ is
strongly $3$-embedded in $G$, $C_P(t)$ contains a Sylow $3$-subgroup of $C_G(t)$. Thus, as $|C_P(r_1)|_3= 3^4$,
we have $r_1^G\cap P \subseteq r_1^P \cup r_2^P$. Since $r_1$ and $r_2$ are not $G$-conjugate by
Lemma~\ref{fusionr1r2} and \ref{index2}, we get that $r_1^G \cap P= r_1^P$.  We note that if $N$ is a non-trivial normal subgroup of $G$, then, as $C_G(r_1) \le P$ and $r_1 \not \in Z(P)$, $1 \neq C_N(r_1) \le N \cap P$ which means that $P \le N$. Because  $N_G(S) \le P$, the Frattini Argument implies $G = N_G(S)N \leq PN =N$.  Hence $G$ is a simple group. Now an application of Lemma~\ref{Holt} and the observation that $P$ is neither soluble nor an alternating group yields  $G=P$ and the proof is complete.
\end{proof}

\end{document}